\documentclass[a4paper]{article}
\usepackage[all]{xy}\usepackage[latin1]{inputenc}        
\usepackage[dvips]{graphics,graphicx}
\usepackage{amsfonts,amssymb,amsmath,color,mathrsfs, amstext}
\usepackage{amsbsy, amsopn, amscd, amsxtra, amsthm,authblk}
\usepackage{enumerate,algorithmic,algorithm}
\usepackage{upref}
\usepackage{geometry}
\usepackage[displaymath]{lineno}
\usepackage{float}
\usepackage{yhmath}
\usepackage{booktabs}
\usepackage{subcaption}
\usepackage{multirow}
\usepackage{makecell}
\usepackage[colorlinks,
            linkcolor=red,
            anchorcolor=red,
            citecolor=red
            ]{hyperref}

\numberwithin{equation}{section}  

\def\b{\boldsymbol}

\def\R{\mathbb{R}}
\def\P{\mathbb{P}}
\def\E{\mathbb{E}}
\def\cL{\mathcal{L}}

\def\cS{\mathcal{S}}

\newtheorem{theorem}{Theorem}[section]
\newtheorem{lemma}{Lemma}[section]
\newtheorem{proposition}{Proposition}[section]
\newtheorem{remark}{Remark}[section]
\newtheorem{corollary}{Corollary}[section]

\newtheorem{assumption}{Assumption}[section]

\makeatletter
\@mparswitchfalse
\makeatother

\title{Ergodicity and error estimate of laws for a random splitting Langevin Monte Carlo}
\date{}

\begin{document}

\author[a,b]{
Lei Li \thanks{Email: leili2010@sjtu.edu.cn}
}
\author[a]{
Chen Wang \thanks{Email: wangchen6326@sjtu.edu.cn}
}
\author[a,b]{
Mengchao Wang \thanks{Email: mc666@sjtu.edu.cn}
}

\affil[a]{School of Mathematical Sciences, MOE-LSC, Shanghai Jiao Tong University, Shanghai 200240, China
}
\affil[b]{Institute of Natural Sciences, Shanghai Jiao Tong University, Shanghai 200240, China}

\maketitle
\begin{abstract}
The random splitting Langevin Monte Carlo could mitigate the first order bias in Langevin Monte Carlo with little extra work compared other high order schemes. We develop in this work an analysis framework for the sampling error under Wasserstein distance regarding the random splitting Langevin Monte Carlo. First, the sharp local truncation error is obtained by the relative entropy approach together with the explicit formulas for the commutator of related semi-groups. The necessary pointwise estimates of the gradient and Hessian of the logarithmic density are established by the Bernstein type approach in PDE theory. Second, the geometric ergodicity is established by accommodation of the reflection coupling. Combining the ergodicity with the local error estimate, we establish a uniform-in-time sampling error bound, showing that the invariant measure of the method approximates the true Gibbs distribution with $O(\tau^2)$ accuracy where $\tau$ is the time step. 
Lastly, we perform numerical experiments to validate the theoretical results. 
\end{abstract}

\section{Introduction}

Sampling is a fundamental challenge in  Bayesian statistics \cite{gelman1995bayesian,andrieu2003introduction,jing2024machine}, statistical physics and computational chemistry \cite{allen2017computer,frenkel2023understanding}, as well as generative AI's \cite{rezende2015variational,songscore,lipmanflow}. The most routinely used approach is the class of Markov chain Monte Carlo (MCMC) methods, consisting of the classical Metropolis-Hastings algorithm \cite{metropolis1953equation,hastings1970monte}, the Hamitlonian Monte Carlo \cite{neal2011mcmc,robert2004monte} and its various generalizations \cite{nishimura2020discontinuous,li2023splitting,li2025generalized}, and the samplers based on Langevin dynamics \cite{rossky1978brownian,roberts1996exponential,leimkuhler2015molecular,li2025second}.
Choosing the suitable proposal distribution is tricky in the Metropolis-Hastings algorithm, so the Hamitonina Monte Carlo and the samplers based on Langevin dynamcis are preferred for sampling from high-dimensional probability distributions. The overdampled Langevin diffusion is a continuous dynamics that leaves the Gibbs distribution \(\rho_{*}(x) \propto e^{-\beta U(x)}\) invariant  
\begin{equation}\label{eq:SDE}
	dX_t = -\nabla U(X_t)\,dt + \sqrt{2\beta^{-1}}\,dW_t,
\end{equation}
where \(U: \mathbb{R}^d \to \mathbb{R}\) is the potential, \(\beta > 0\) is the inverse temperature, and \(W_t\) is a \(d\)-dimensional Brownian motion. The density \(\rho(t,x)\) satisfies the Fokker-Planck equation  
\begin{equation}\label{eq:FPE}
	\partial_t \rho = \nabla \cdot (\nabla U\,\rho) + \beta^{-1} \Delta \rho.
\end{equation}
Using the Fokker-Planck equation, it is relatively easy to verify that the Gibbs distribution \(\rho_{*}(x) \propto e^{-\beta U(x)}\) is invariant for the dynamics. Under mild regularity conditions, the law of \(X_t\) converges to the invariant Gibbs distribution. A common discretization of \eqref{eq:SDE} is the Euler-Maruyama scheme, leading to the Langevin Monte Carlo (LMC) algorithm \cite{rossky1978brownian,roberts1996exponential}:
\begin{equation}\label{eq:Euler}
	X_{n+1} = X_n - \eta_n \nabla U(X_n) + \sqrt{2\beta^{-1}\,\eta_n}\,\Delta W_n.
\end{equation}
However, this scheme generally incurs a first-order bias in the invariant distribution \cite{talay1990expansion,durmus2015quantitative,wu2022minimax}. To mitigate this bias and improve long-time accuracy, a number of strategies have been proposed.

One classical approach is to utilize variable vanishing stepsizes. A typical scheduling is to choose $\eta_n$ such that $\sum_n \eta_n=\infty$ and $\sum_n \eta_n^2<\infty$, known as the Robbins--Monro condition from stochastic approximation \cite{robbins1951stochastic,kushner2003stochastic,teh2016consistency}.
Another well-known remedy is to append a Metropolis--Hastings correction. The Metropolis-adjusted Langevin algorithm (MALA) \cite{roberts1996exponential,xifara2014langevin,chewi2021optimal} achieves asymptotically exact sampling and has been extensively studied for its mixing and scaling properties \cite{atchade2006adaptive,roberts1998optimal}. However, the accept-reject step incurs computational overhead and complicates parallelization, motivating interest in alternatives that attain higher accuracy without rejection.

Alternatively, one may rely on some higher-order schemes for the SDEs. The Lie--Trotter and Strang splittings have been applied to reduce the bias and improve the weak accuracy \cite{ninomiya2008weak,ninomiya2009new,leimkuhler2013rational,hartmann2016molecular}. Higher-order compositions can improve invariant-measure accuracy \cite{abdulle2015long,bernton2018langevin}. Complementary estimator-level techniques such as multilevel Monte Carlo have also been studied by Anderson et al.~\cite{anderson2012multilevel}, which can further reduce bias and variance and naturally combine with higher-order weak schemes.  While such methods can reduce bias to $O(\tau^2)$ or beyond, most of them (except few examples like the one in \cite{leimkuhler2013rational}) may limit robustness in complex or high-dimensional settings and suffer from high cost as each iteration needs several evaluations of gradients.

A simple strategy to improve the sampling error while keeping the simplicity is to apply the Lie--Trotter splitting in a random order (see a related work in \cite{cho2024doubling} for the Schr\"odinger operators). Intuitively, the first order bias could cancel in the long time simulations due to the law of large numbers as illustrated in the analysis of such methods for PDE simulations \cite{li2025convergence}. 
Compared to the MALA approach or the higher order schemes, such strategy is a slight modification of the classical Langevin Monte Carlo and thus enjoys better robustness and low cost.  
A separate line of randomization introduces stochasticity to improve scalability and efficiency for Bayesian inference with large datasets \cite{welling2011bayesian,raginsky2017non} and for interacting particle systems \cite{jin2020random,li2020random}. 
The combination of such randomization with the Hamiltonian dynamics and the Langevin dynamics has turned to be useful in improving the efficiency of sampling for many-body problems \cite{li2020random,liang2022random,li2025generalized}, with the price of some acceptable systematic errors.


While the fact that the randomized splitting could improve accuracy is intuitively clear, the rigorous analysis of its long-time behavior and the error bound, in particular under Wasserstein distances, have not been systematically established. Our focus in this paper goes to the simplest random Lie--Trottter splitting method, named as randomized splitting Langevin Monte Carlo (RSLMC) method, where drift and diffusion substeps are applied in random order at each iteration (see section \ref{subsec:rslmc} for details). Our goal in this paper is to propose an analysis framework for the quantitative error bound for RSLMC under Wasserstein distance.  We remark that the quantitative error analysis for laws of numerical schemes of SDEs and sampling methods under certain metrics (e.g., Wasserstein distances, total variation distance) becomes a central topic in recent years. Compared to the traditional weak error analysis where the bound depends on the particular choice of smooth test function \cite{milstein2004stochastic,wang2024weak}, such error estimates are applicable to the error bounds for a class of nonsmooth test functions \cite{li2024estimates,durmus2019nonasymptotic,vempala2019rapid}.  The relative entropy (i.e.\ Kullback--Leibler divergence) has emerged as a particularly powerful tool for quantifying these error bounds. Early works such as \cite{cheng2018convergence,dalalyan2019user,durmus2019high} established non-asymptotic convergence bounds for LMC-type methods, clarifying the dependence of approximation error on the step size and the ambient dimension.  More recent advances include \cite{parulekar2025efficient} where the annealed Langevin Monte Carlo with KL guarantee for posterior sampling has been proposed, and \cite{zhang2025analysis} where the midpoint discretization has been analyzed via anticipative Girsanov techniques. It is noted that the relative entropy has recently been proved useful to obtain the sharp error bounds using some PDE techniques and eventually yield bounds under the Wasserstein distances using the transport inequalities. The  $O(\tau^2)$ bound of relative entropy for the Langevin Monte Carlo with additive noise has been obtained in Mou et al.~\cite{mou2022improved}, and the same bound has been established for the SGLD in~\cite{li2022sharp}. By combining the Malliavin calculus, the same bound of relative entropy has recently been proved in \cite{li2024estimates} for the Euler discretization of SDEs with multiplicative noise.

Our contribution in this work can be summarized as follows. We first establish the sharp error bounds for  RSLMC using the relative entropy for finite time horizon inspired by the previous works \cite{mou2022improved,li2022sharp,li2024estimates} while the necessary pointwise estimates of the gradient and Hessian of the logarithmic density are established by the Bernstein type approach in PDE theory. Secondly, we establish the ergodicity of  RSLMC under the Wasserstein-1 distance using the reflection coupling
following \cite{eberle2016reflection,li2024geometric}. Lastly, combining the sharp estimate for finite time and the ergodicity, we are able to the obtain the uniform-in-time sharp error estimate for the method and thus the sharp sampling error for  RSLMC: there exists a unique invariant distribution $\bar{\rho}_*^\tau$ of RSLMC, and it satisfies 
\[
W_1(\bar{\rho}_*^\tau, \rho_*) \le C\, \tau^2,
\]
where $\rho_*$ is the desired Gibbs measure and $W_1$ is the Wasserstein-1 distance.  Note that the pointwise estimates of gradients and Hessian of the logarithmic numerical density serve primarily as technical tools in our proofs but may be of independent interest. The analysis framework in this paper could be possibly generalized to other similar sampling methods to improve the existing bounds in literature like the implicit Langevin Monte Carlo, the midpoint method and other high order splitting schemes.

The remainder of the paper is organized as follows. Section~\ref{sec:setup} introduces the RSLMC algorithm and the main assumptions. Section~\ref{sec:ergodicity} establishes ergodicity and exponential convergence. Section~\ref{sec:error-estimate} presents finite-time relative entropy bounds, supported by PDE-based density estimates. Section~\ref{sec:sampling-error} combines local error analysis with ergodicity to derive uniform-in-time error bounds. Section~\ref{sec:experiments} reports numerical experiments validating the theoretical insights. Technical proofs and auxiliary results are deferred to the appendices.

\section{Setup and some basic results}\label{sec:setup}

As mentioned above, the Langevin Monte Carlo introduces a first order bias of the order $O(\tau)$ in the invariant distribution. 

As mentioned, we propose a \textit{random splitting method} that alternates between the drift ($\dot{X} = -\nabla U(X)$) and diffusion ($dX = \sqrt{2\beta^{-1}}dW$) steps in random orders. This randomization cancels low-order bias terms asymptotically, enabling higher-order accuracy. Note that the random splitting approach can be generalized to general splitting with $p$ terms for $p\ge 2$, and can be applied to other types of SDEs like the underdamped Langevin equation.
In this work, we stick to this simple setting and propose an analysis framework for the error estimate. Our goal is to show that the splitting method has a second order sampling accuracy under some metric. 
In particular, under the Wasserstein distances. The multiplicative noise case can be generalized to as well, but is much more involved.

To improve the order one may apply some second order schemes, for example, the symmetrized Trotter splitting can achieve $O(\tau^2)$ accuracy but it requires evaluation of dynamics for multiple times in  one iteration.

\subsection{The random splitting Langevin Monte Carlo}\label{subsec:rslmc}

As mentioned, for the random splitting Langevin Monte Carlo, we split the SDE into an ordinary differential equation (ODE) and a pure Brownian diffusion, namely
\begin{gather}\label{eq:splitODE}
\dot{X}=-\nabla U(X)
\end{gather}
and
\begin{gather}\label{eq:splitBM}
dX=\sqrt{2\beta^{-1}}\,dW.
\end{gather}
The random splitting approach is to solve these two parts consecutively in a random order for each iteration. The ODE part can be solved
using any standard method that is at least second order, like the Runge-Kutta method of second order accuracy. Let such a solver with initial state $x$ be denoted by
\begin{gather}
X(\tau)=: S(x, \tau).
\end{gather}
The diffusion part is 
to add simply a standard normal variable to the process. The detailed method is given in Algorithm \ref{alg:randomsplit}.
\begin{algorithm}[!ht]
	\caption{The random splitting Langevin Monte Carlo}
	\label{alg:randomsplit}
	\begin{algorithmic}[1]
		\REQUIRE 
		Time step $\tau>0$ , terminal time $T>0$. 
	\STATE Generate initial data $X_0$ from some initial distribution.
		
	\FOR{$n=0: \lceil T/\tau\rceil-1$}
	\STATE Generate a random number $\zeta^{n}\sim \mathrm{U}[0, 1]$, and $Z_n\sim \mathcal{N}(0, I)$.
	\IF{$\zeta^n\le 1/2$}
	\STATE
	\begin{gather}
		\bar{X}_{n+1}=S(X_{n}, \tau),\quad X_{n+1}=\bar{X}_{n+1}+\sqrt{2\beta^{-1}\tau}Z_n.
	\end{gather}
	\ELSE
	\STATE 
	\begin{gather}
	\bar{X}_{n+1}=X_n+\sqrt{2\beta^{-1}\tau}Z_n, 
	\quad X_{n+1}=S(\bar{X}_{n+1}, \tau).
	\end{gather}
	\ENDIF
	\ENDFOR
	\end{algorithmic}
\end{algorithm}

Below, we will assume that $S(x, \tau)$ is the accurate ODE evolution. If we replace with the discretization, the global error will be controlled simply using the triangle inequality.  
For notational convenience, we denote
\begin{gather}
b(x):=-\nabla U(x),
\end{gather}
and the density evolution for the ODE is then given by
\begin{gather}\label{eq:FP}
\partial_t\rho+\nabla\cdot(b(x)\rho)=0.
\end{gather}
Clearly, the operator corresponding to the ODE part can be expressed as 
\begin{gather}\label{eq:operaL1}
\cL_{1}(\rho)=-\nabla\cdot(b(x)\rho).
\end{gather}
Similarly, the density evolution for the diffusion step is
\begin{gather}\label{eq:operaL2}
\partial_t\rho=\beta^{-1}\Delta\rho=:\cL_{2}(\rho).
\end{gather}
The Fokker-Planck operator for the Langevin equation is clearly
\begin{gather}
\cL\rho=\cL_{1}(\rho)+\cL_{2}(\rho)=-\nabla\cdot(b(x)\rho)+\beta^{-1}\Delta\rho.
\end{gather}

To estimate the local truncation error, we need the expressions of the semigroup. Let $\phi_t(x)$ denote the flow map for the ODE, which is the trajectory of the ODE with initial position $x$:
\begin{gather}
\partial_t\phi_t(x)=b(\phi_t(x)),\quad \phi_0(x)=x.
\end{gather} 
Moreover, we introduce the determinant of the Jacobi matrix
\begin{gather}
J(t, x):=\det(\nabla_x \phi_t(x)).
\end{gather}
Along the ODE, it is easy to see that 
\begin{gather}
\frac{\partial}{\partial t}\nabla_x\phi_t(x)=\nabla_x\phi_t(x)\cdot \nabla b(\phi_t(x)),
\end{gather}
where we use the convention that
\[
(\nabla b)_{ij}=\partial_i b_j(x).
\]
It then follows that
\begin{gather}\label{eq:Jacobieq}
\partial_t J(t, x)=\text{tr}( (\nabla_x\phi)^{-1} \partial_t(\nabla_x\phi)) J(t,x)=(\nabla\cdot b)(\phi_t(x)) J(t, x).
\end{gather}

We note the following simple fact.
\begin{lemma}\label{lem:evolution-opera}
Let $\rho$ denote the probability density function associated with the SDE \eqref{eq:SDE}. With the operators $\mathcal{L}_1$ and $\mathcal{L}_2$ defined in \eqref{eq:operaL1} and \eqref{eq:operaL2}, respectively, then the semigroups are given as follows:
\begin{gather}
 e^{t\cL_{1}} \rho (x) = \rho ( \phi_{-t} (x) ) J(-t,x),
\end{gather}
and that
\begin{gather}
 e^{t\cL_{2}} \rho (x) =\int_{\R^d} \frac{1}{(4 \pi \beta^{-1} t)^{\frac{d}{2}}}
e^{-\frac{\beta |x-y|^2}{4  t} } \rho (y) dy.
\end{gather}
\end{lemma}

\begin{proof}
Consider the semigroup $e^{t\cL_{1}}$. Denote $\varrho(x, t)=e^{t\cL_{1}}\rho$. Then,
\begin{gather}
\partial_t\varrho+b\cdot\nabla \varrho=-(\nabla\cdot b)\varrho.
\end{gather}
Then, along the characteristics, it is easy to see that
\begin{gather}
\varrho(\phi_t(x), t)=\varrho(x, 0)\exp\left( -\int_0^t \nabla\cdot b(\phi_s(x))\,ds\right).
\end{gather}
By \eqref{eq:Jacobieq}, one has
\[
(e^{t\cL_{1}}\rho)(\phi_t(x))=\rho(x)J(t, x)^{-1}.
\]
Using the property of evolution group,
\[
\phi_{-t}(\phi_t(x))=x \quad \Longleftrightarrow  \quad J(-t, \phi_t(x)) J(t, x)=1.
\]
It then follows that
\begin{gather}
(e^{t\cL_{1}}\rho)(\phi_t(x))=\rho(x)J(-t, \phi_t(x)).
\end{gather}
The expression for the first semigroup then follows.

Since $\cL_{2}$ generates the standard heat semigroup on $\mathbb{R}^d$, its action is explicitly given by convolution with the Gaussian heat kernel:
\begin{gather}
 e^{t\cL_{2}} \rho (x) =\int_{\R^d} \frac{1}{(4 \pi \beta^{-1} t)^{\frac{d}{2}}}
e^{-\frac{\beta |x-y|^2}{4  t} } \rho (y) dy.
\end{gather}
This representation follows directly from the Fourier transform solution to the heat equation $\partial_t \rho = \beta^{-1}\Delta \rho$ with initial data $\rho(0,\cdot)=\rho(\cdot)$. We therefore obtain the result stated in the Lemma \ref{lem:evolution-opera}.
\end{proof}

For a given random variable $\zeta$, there is a corresponding random permutation $\xi$ of $\{(1,2),(2,1)\}$. The probability of each permutation is $1/2$. The density evolution is then given by
\begin{gather}\label{eq:rhoxi}
\rho_{n+1}^{\xi}=e^{\tau \cL_{\xi(2)}}e^{\tau \cL_{\xi(1)}}\rho_{n}^{\xi},
\end{gather}
where the operators are applied in random order at each step.
The law of $X_{n+1}$ is then
\[
\bar{\rho}_{n+1}=\E_{\xi}\rho_{n+1}^{\xi}.
\]
Here, we consider two versions of the time-continuous density. The first is
\begin{gather}\label{eq:timeinterpolation0}
\rho^{\xi}(t)=e^{(t-n\tau)\cL_{\xi(1)}}e^{(t-n\tau)\cL_{\xi(2)}}\rho_n^{\xi},
\quad t\in [n\tau, (n+1)\tau],
\end{gather}
while the second is 
\begin{gather}\label{eq:timeinterpolation}
\rho^{\xi_n}(t)=e^{(t-n\tau)\cL_{\xi(1)}}e^{(t-n\tau)\cL_{\xi(2)}}\bar{\rho}_n,
\quad t\in [n\tau, (n+1)\tau].
\end{gather}
By the linearity of the operators, one clearly has
\[
\rho^{\xi_n}(t)=\E_{\xi_1,\cdots, \xi_{n-1}}[\rho^{\xi}(t)|\xi_n].
\]
That means the average is taken by considering all the random permutations for $t_i, i\le n-1$.
We then define
\begin{gather}\label{eq:barrhot}
\bar{\rho}(t)=\mathbb{E}_{\xi_n}\rho^{\xi_n}(t)=\E_{\xi}\rho^{\xi}(t).
\end{gather}

Since $\bar{\rho}(t_n)=\bar{\rho}_n$ which is the law of $X_n$.
The key quantity of interest is the approximation error between $\bar{\rho}(t)$ and the law of the original Langevin dynamics $\rho$.
To quantify the approximation error, we use the 
relative entropy (also Kullback-Leibler (KL) divergence). The relative entropy between two probability measures $\mu,\nu$ are defined by
\begin{gather}\label{eq:kl_divergence}
	\mathcal{H}(\mu\,\|\,\nu)
	:=\begin{cases}
		\displaystyle \int \log\!\Big(\frac{d\mu}{d\nu}\Big)\,d\mu, & \mu \ll \nu,\\[0.8ex]
		+\infty, & \text{otherwise},
	\end{cases}
\end{gather}
where $\tfrac{d\mu}{d\nu}$ denotes the Radon--Nikodym derivative, which rigorously captures both mass displacement and tail behavior discrepancies.
In particular, if $\mu$ and $\nu$ admit densities $p,q$ with respect to Lebesgue measure, then
\[
\mathcal{H}(\mu\|\nu) \;=\; \int_{\R^d} p(x)\,\log\!\frac{p(x)}{q(x)}\,dx.
\]
It is known that the relative entropy has the following properties (see \cite{li2024geometric,thomas2006elements})
\begin{enumerate}
\item \emph{Nonnegativity:} $\mathcal{H}(\mu\|\nu)\ge0$, with equality if and only if $\mu=\nu$.
\item \emph{Joint convexity:} The map $(\mu,\nu)\mapsto \mathcal{H}(\mu\|\nu)$ is jointly convex.  
\item \emph{Lower semicontinuity:} For fixed $\nu$, the map $\mu\mapsto\mathcal{H}(\mu\|\nu)$ is lower semicontinuous with respect to weak convergence of measures.
\end{enumerate}
The relative entropy is useful because it can be used to control the Wasserstein distances through the transport inequalities.
We recall that the Wasserstein distances between two probability measures $\mu,\nu\in \mathcal{P}(\R^d)$ are defined by
\[
W_p(\mu,\nu) := \inf_{\pi \in \Gamma(\mu,\nu)}
\left( \int_{\R^d \times \R^d} |x-y|^p \, d\pi(x,y) \right)^{1/p},\quad p\ge 1,
\]
where $\Gamma(\mu,\nu)$ denotes the set of all couplings of $\mu$ and $\nu$.

\begin{enumerate}
		\item[\textnormal{(T\textsubscript{1})}] If $\nu$ satisfies sub-Gaussian concentration for all 1-Lipschitz functions,
        then there exists $C>0$ such that
		\[
		W_1(\mu,\nu) \;\le\; \sqrt{\,2C\,\mathcal{H}(\mu\|\nu)}.
		\]
		\item[\textnormal{(T\textsubscript{2})}] If $\nu$ satisfies a logarithmic Sobolev inequality with constant $C$, then
		\[
		W_2(\mu,\nu) \;\le\; \sqrt{2C\,\mathcal{H}(\mu\|\nu)}
		\]
	\end{enumerate}
	Moreover, since one has $W_1\le W_2$, hence \(T_2(C)\Rightarrow T_1(C)\).

Our goal is to prove error bounds of the form $\mathcal{H}(\bar{\rho}(t) \| \rho(t)) \leq C(T)\tau^4$ under appropriate assumptions.
Later, we will obtain a uniform subGaussian estimate for $\rho(t)$ so that $\rho(t)$ satisfies a uniform \(T_1(C_T)\) inequality. Then, the bound on relative entropy can give a sharp bound on the Wasserstein-1 distance of the order $O(\tau^2)$.





\subsection{Some assumptions and basic results}
We begin by specifying the regularity assumptions on both the  potential function $U(x)$ and the initial distribution  required for our analysis. The main result of this subsection establishes the uniform-in-time polynomial bounds on the gradient and Hessian of the log-density, which play a pivotal role for the uniform-in-time error bound.

We will work under the following Lipschitz assumptions, which are standard for Langevin Monte Carlo analysis.
\begin{assumption}\label{ass:drift}
The potential function $U:\R^d\to \R$ is $C^4$ and there exists $M>0$ such that when $|x|\ge M$, $\nabla^2 U\succeq \kappa I$ for some $\kappa>0$. Moreover,  $b(x)=-\nabla U(x)$ is Lipschitz in the sense that
\[
|b(x) - b(y)| \leq L|x - y|, \quad \forall x, y \in \mathbb{R}^d,
\] 
and the second order and third order derivatives of $b$ are bounded.
\end{assumption}

By the assumption, one has the following simple observation.
\begin{lemma}\label{lmm:onvexity}
Under Assumption \ref{ass:drift}, the drift $b$ has at most linear growth $|b(x)| \leq C(1 + |x|)$ for some constant $C > 0$, and there exists $R>0$
such that whenever $|x-y|\ge R$,
\begin{gather*}
(x-y)\cdot (b(x)-b(y))\le -\lambda |x-y|^2\quad \Longleftrightarrow \quad (x-y)\cdot (\nabla U(x)-\nabla U(y))\ge \lambda |x-y|^2.
\end{gather*}
\end{lemma}

\begin{lemma}\label{lem:moment_hybrid_refined}
\begin{enumerate}[(i)]
\item 
Let $\rho(t,\cdot)$ solve \eqref{eq:FPE}.
For every $p\ge2$ there exist constants $\lambda>0$ (depending only on the coefficients and $\beta$, but independent of $p$) and $C_p>0$ such that, for all $t\ge0$,
\[
\int_{\mathbb{R}^d} |x|^p\,\rho(t,x)\,dx \;\le\; e^{-\lambda t}\!\int_{\mathbb{R}^d} |x|^p\,\rho_0(x)\,dx \;+\; C_p .
\]

\item 
Let $\{\rho_n^\xi\}_{n\ge0}$ be the discrete solution with stepsize $\tau>0$, where at each step $\xi_i\in\{(1,2),(2,1)\}$ is chosen independently with probability $1/2$ each. Then there exist $\tau_0>0$ and $C_p>0$ such that, for all $\tau\in(0,\tau_0)$,
\[
\sup_{n\ge0}\, \int_{\mathbb{R}^d} |x|^p\,\rho_n^\xi(x)\,dx
\;\le\; C_p\Big(1+\int_{\mathbb{R}^d} |x|^p\,\rho_0(x)\,dx\Big),
\]
where the constant $C_p$ is uniform in $\xi$.
\end{enumerate}

\end{lemma}

The proofs of Lemma \ref{lmm:onvexity} as well as part (i) of Lemma \ref{lem:moment_hybrid_refined} are standard. We omit the proofs and refer the readers to \cite[Appendix A]{li2024geometric} for related arguments.
Below, we present a proof for Lemma \ref{lem:moment_hybrid_refined} (ii).

\begin{proof}[Proof of Lemma \ref{lem:moment_hybrid_refined} (ii)]

Denote
\[
M_n^p:=\int_{\R^d} |x|^p \rho_n^{\xi}\,dx.
\]
Consider the diffusion step first. Let $\varrho_s=e^{s\cL_{\xi(2)}}\varrho_0$
and $M_s:=\int |x|^p \varrho_s(x)\,dx$. Then,
\[
\partial_s M_s=\int \beta^{-1}|x|^p \Delta \rho_s\,dx
=\beta^{-1}p(p-2+d)\int |x|^{p-2}\rho_s\,dx.
\]
For any $\varepsilon\in(0,1]$, Young's inequality gives $|x|^{p-2}\le \varepsilon |x|^p+C(\epsilon,p)$. Hence, one finds for $\tau\le 1$ that
\[
M_{\tau}\le e^{C\epsilon \tau}M_0+C\tau.
\]

For the advection step, we define similarly (with abuse of the notations)
that $\varrho_s=e^{s\cL_{\xi(1)}}\varrho_0$ and $M_s:=\int |x|^p \varrho_s(x)\,dx$. Then,
\[
\partial_sM_s=\int p |x|^{p-2}x\cdot b(x)\varrho_s\,dx.
\]
By the assumption on $b$, one has
\[
x\cdot b(x)\le -\lambda |x|^2+x\cdot b(0)
\le -\frac{3\lambda}{4}|x|^2+C.
\]
Then, one has for $\tau\le 1$ that
\[
M_{\tau}\le M_0\exp(-\frac{3\lambda}{4}\tau)+C'\tau.
\]

Combining the two substeps in either order given, for some $\tau_0>0$ and all $0<\tau\le\tau_0$, it is easy to see that 
\[
M_{n+1}^p \le M_n^p\exp(-\frac{3\lambda}{4}\tau)e^{C\epsilon \tau}+C\tau.
\]
Clearly, one can choose $\epsilon$ such that
\[
M_{n+1}^p \le M_n^p\exp(-\frac{\lambda}{2}\tau)+C\tau.
\]
Iterating this gives the asserted uniform-in-time bound in \textup{(ii)}.
\end{proof}

\begin{assumption}\label{ass:initial-density}
The initial density $\rho_0$ is assumed to satisfy the following conditions. 
First, it satisfies the following bounds:
\[
A_0 \exp(-c_0|x|^p) \;\le\; \rho_0(x) \;\le\; A_1 \exp(-\delta\beta U(x))
\]
for some constants $A_1, A_2, c_0, \delta > 0$ and $p > 0$.
Moreover, the derivatives of the logarithm of $\rho_0$ satisfy:
\[
|\nabla \log \rho_0(x)| \le C \bigl(1 -\log(\rho_0/M)\bigr), 
\qquad
\|\nabla^2 \log \rho_0(x)\| \le C \bigl(1 + |x|^{q}\bigr),
\]
for some $M\ge \|\rho_0\|_{\infty}$ and exponent $q>0$.
\end{assumption}

With the assumption above, we have the following estimate.
\begin{proposition}\label{pro:density-SDE}
There exists some constant such that
\[
A_2\exp(-\gamma_1 |x|^p)\le \rho(x, t)\le C[\exp(-\delta\beta U(x))e^{-\lambda t}+\exp(-\beta U(x))].
\]
\end{proposition}
The proof is provided in Appendix \ref{app:esti-FP}.

\begin{theorem}[Derivative estimates for log-density]\label{thm:gradient_estimate}
Under Assumptions \ref{ass:drift} and 
\ref{ass:initial-density},  the Fokker-Planck solution $\rho(t,x)$ satisfies for some $p_1>0$, $p_2>0$, $C_1>0,C>0$, $M\ge \|\rho\|_{L^{\infty}(\R^d\times [0,\infty))}$ and all $t\ge 0$ that:
\begin{equation}\label{eq:gradient_bound}
|\nabla \log\rho(x,t)|\le C_1(1-\log(\rho(x,t)/M))  \leq C ( 1 + |x|^{p_1}),
\end{equation}
\begin{equation}\label{eq:gradient_bound2}
 \|\nabla^2 \log \rho(x,t)\| \leq C ( 1 + |x|^{p_2}) .
\end{equation}
\end{theorem}
\begin{proof}
In the proof below, we will use $C$ to denote a generic constant that is independent of the parameters in consideration. The concrete meaning of $C$ can change from line to line.

\vskip 0.1 in

\noindent \textbf{First order derivative estimates.}\\
  As we have shown, for all $t\ge 0$, there is a uniform upper bound for $\rho(t, x)$. Let one such upper bound be $M$. 
Let $ f = \log \rho$ and we consider
\begin{gather}
v:=\frac{|\nabla f|^2}{(1-f+\log M)^2}=\frac{f_k f_k}{(1-f+\log M)^2}.
\end{gather}
In the proof here, we use the Einstein summation convention so that summation is performed for repeated indices.
Note that
\[
1-f+\log M\ge 1.
\]
One has by Assumption \ref{ass:initial-density} that (if the constant $M$ does not agree with the one in Assumption \ref{ass:initial-density}, one may enlarge the bound slightly)
\[
\sup_x v(x, 0)<\infty.
\]

Since the equation for $\rho$ is linear, without loss of generality, we may take $M=1$ for convenience of the discussion in this section. 
The equation for $f=\log\rho$ (or $\log(\rho/M)$) is given by
\[
f_t= \beta^{-1} f_{i i}+ \beta^{-1} f_i f_i - b^i f_i+c ,
\]
where $ c = \Delta U $.
Define
\[
\mathcal{L} v:=  \beta^{-1} v_{i i} - b^i v_i-v_t .
\]
In the following calculations, without loss of generality, we may take 
\[
\beta = 1
\]
 for convenience in this section. 
Then, one can derive the following estimate
\[
\tilde{\cL}v:=\cL v-\frac{2  f}{1-f}   f_i v_i  \geq \frac{2}{3}  \frac{f_{ki}f_{ki}}{(1-f)^2}+  (1-f) v^2 - R_1(x),
\]
where
$$
\begin{aligned}
	R_1(x)=  \frac{2 c f_k f_k}{(1-f)^3}+2 \frac{c_k f_k-b_k^i f_k f_i}{(1-f)^2} .
\end{aligned}
$$
The details are given in Appendix \ref{app:compute}.

Fix $x_*$ and take the cutoff function $\phi(x)=\psi(|x-x_*|/(|x_*| + 1))$ with $\psi$ satisfying the following: (1) $\psi \in C^2[0, \infty)$ such that
$0 \leq \psi(r) \leq 1$, and is  $1$ on $[0, 1]$, $0$ for $r\ge 2$; (2)
\begin{gather}
\psi'\le 0, \quad |\psi'|+|\psi''|\le C_{a}\psi^{a},\quad \forall a\in (0, 1).
\end{gather}
Then, $\phi$ is supported in $B(x_*, 2|x_*|+2)$ and
\[
\tilde{\mathcal{L}}(\phi v)=\phi \tilde{\mathcal{L}} v + v \tilde{\mathcal{L}} \phi + 2   \phi_i v_i.
\]
From the estimates obtained, one has 
$$
\phi \tilde{\mathcal{L}} v \geq (1-f) \phi v^2 - \phi R_{1}(x) .
$$
Moreover, due to $ \phi_i(x) = \psi'\left( \frac{|x - x_*|}{|x_*| + 1} \right) \cdot \frac{x_i - x_{*i}}{(|x_*| + 1) |x - x_*|}$ and $ |b(x)| \le C (1 + |x|)$, we find that
\[
\begin{split}
v \tilde{\mathcal{L}} \phi &=
	  v \phi_{ii} - b^{i} v \phi_{i}  -\frac{2  f}{1-f}   v f_i \phi_i 
	\ge C v(-1-  \frac{|b(x)| \mathbf{1}_{|x_*|+1 \le|x-x_*|\le 2|x_*|+2}}{|x_*| +1}- \frac{|f||f_i||\phi_i|}{1-f} )\\
&\geq -C v\left(1+\epsilon |\phi_i|^2(1-f)v+\frac{f^2}{1-f} \right) ,
\end{split}
\]
and
\[
2  \phi_i v_i  = 2   \phi_i \phi^{-1} \left( \phi v\right)_i 
	- 2 \phi_i^2\phi^{-1} v .
\]
 Using the condition $|\psi'| \leq C_a \psi^a$ with $a = \frac{1}{2}$, one has
\[
|\phi_i|^2 \leq |\psi'(r)|^2 \cdot \frac{1}{(|x_*| + 1)^2}
\leq C^2 \phi \cdot \frac{1}{(|x_*| + 1)^2} \le C \phi.
\]
Hence, one has
\begin{gather}
\tilde{\mathcal{L}}(\phi v) \geq (1-C\epsilon)(1-f)\phi v^2-C v\left[1+\frac{f^2}{1-f}\right] +2\phi_i\phi^{-1}(\phi v)_i - M_1 \phi,
\end{gather}
where
$$
\begin{aligned}
M_1 =  C_{\varepsilon} \big( 1 
+\|\partial b \|_{L^{\infty}}^2  +\|c\|_{L^{\infty}}^2
+\|\partial c\|_{L^{\infty}}^2 
\big).
\end{aligned}
$$
It is then left to consider where the function $\phi v(x, t)$ attains its maximum.
\begin{itemize}
	\item Since $\phi$ vanishes for $|x-x_*|\ge 2(|x_*|+1)$, if $\phi v$ reaches its maximum on the boundary of $[0,\infty)\times B(x_*, 2(|x_*|+1))$, then it must be the initial time line so that 
	\[
	v(x_*, t)=(\phi v)(x_*, t)\le \sup_{x }(\phi v)(x, 0)\le \sup_x v(x, 0).
	\]
	\item If $\phi v$ reaches its maximum at $\left(x', t'\right)$ with $t'>0$, then $x'$ must be in the interior of the ball $B(x_*, 2|x_*|+2)$ so that $\phi_i \phi^{-1}$ is well-defined at this point,  and $(\phi v)_i=0$ at that point. Then, at $(x', t')$, it holds by the parabolic maximum principle that
\[
0 \geq \tilde{\mathcal{L}} ( \phi v\left(x', t'\right) ) 
	\geq 
	(1-C\epsilon)(1-f)\phi v^2-C v\left[1+\frac{f^2}{1-f}\right] - M_1.
\]
If $v(t', x')\le 1$, then one has $v(t, x_*)=\phi(t, x_*)w(t, x_*)\le w(t', x')\phi(t', x')\le 1$. 
Otherwise, 
\begin{equation}
\begin{split}
&v(x_*, t)=\phi(x_*) v(x_*, t)\le \phi(x')v(x', t')\\
\le& \frac{C}{1-C\epsilon}\frac{1}{1-f}\left[1+\frac{f^2}{1-f}\right] + \frac{M_1}{(1 - C\epsilon)(1 - f) v(x',t')} \le  C.   
\end{split}
\end{equation}

\end{itemize}
By the arbitrariness of $x_*$, one then has $\sup_{x,t}v(x,t)<\infty$ so that
\[
|\nabla f| \le C(1-f+\log M)\le C(1+|x|^{p_1}),
\]
for some $p_1>0$ as a consequence of Proposition \ref{pro:density-SDE}.

\textbf{ Second order derivative estimates.}\\
Consider
\begin{gather}
w:=\frac{1}{q(x)}\frac{|\nabla^2 f|^2}{A(1-f+\log M)^2-|\nabla f|^2}=\frac{1}{q}\frac{f_{k\ell} f_{k\ell}}{A(1-f+\log M)^2-f_jf_j},
\end{gather}
\begin{equation}\label{eq:def-u}
\quad
u \;:=\; \frac{|\nabla^2 f|^2}{\,A(1-f+\log M)^2 - |\nabla f|^2\,},
\qquad
w \;=\; \frac{u}{q(x)},
\end{equation}
where $q(x)=1+|x|^{r}$ for some $r\ge 0$ to be determined such that $w$ is bounded at $t=0$ and some extra polynomial terms will be controlled. By the results of $v$ above, we can choose $A$ such that
\begin{gather}\label{eq:aux1}
\frac{A}{2}(1-f+\log M)^2  \le A (1-f+\log M)^2-|\nabla f|^2\le A(1-f+\log M)^2.
\end{gather}
Though $(1-f+\log M)^2$ is comparable to $A (1-f+\log M)^2-|\nabla f|^2$, we do not use the former in the denominator
as we need the derivatives of $f_kf_k$ to help to control certain terms which shall be clear in the proof. As before, without loss of generality, we can set again $M=1$ and $\beta=1$. 

Consider again the operator
\[
\cL u= u_{ii}+ b^i u_i -u_t.
\]
By direct computation (Appendix \ref{app:compute} provides the details), one has the following estimate
\[
\tilde{\cL}u:=\cL u-4\frac{A(1-f)f_i+f_k f_{ki}}{A(1-f)^2-f_jf_j} u_i
\ge \frac{ f_{k\ell i}f_{k\ell i}}{A(1-f)^2-f_j f_j}+\frac{3}{2}u^2 - C q(x)
-R_{2}(x),
\]
where
$$
R_2 := \frac{2f_{k\ell} (f_i b^i_{\ell k} + b^i_\ell f_{i k} + b^i_k f_{i\ell} + c_{kl})}{A(1-f)^2 - f_j f_j} 
+ \frac{2f_{k\ell} f_{k\ell} \left( A(1-f) c + f_j c_j \right)}{(A(1-f)^2 - f_j f_j)^2}.
$$
According to Young's inequality, one has 
$$ |R_2| \leq  C ( \| \partial^2 c \|^2 + \|\partial b \|^2 + \| \partial^2 b \|^2 ) ( u + v ) +
C u v \le C (1 + u).
$$
Then, one has
\begin{multline*}
	\tilde{\cL} w
	=\frac{1}{q}\left(\cL u-4\frac{A(1-f)f_i+f_k f_{ki}}{A(1-f)^2-f_jf_j} u_i\right)
	-2 (\log q)_i w_i- [(\log q)_i^2+(\log q)_{ii}]w\\
	+4\frac{A(1-f)f_i+f_k f_{ki}}{A(1-f)^2-f_jf_j} w(\log q)_i
	 -  b^{i} (\log q )_{i} w.
\end{multline*}
Note that the derivatives of $\log q$ and
$b^{i} (\log q )_{i} =\frac{r |x|^{r-2} (b \cdot x)}{1 + |x|^{r}} $
  are bounded. 
Moreover, according to Young's inequality and the results of first-order derivative estimates, it follows that
\begin{gather}\label{auxeq:transcoefficient}
	\left|\frac{A(1-f)f_i+f_k f_{ki}}{A(1-f)^2-f_jf_j}\right|
	\le C+\epsilon u.
\end{gather}
Then, we obtain
\[
\cL w+[-4\frac{A(1-f)f_i+f_k f_{ki}}{A(1-f)^2-f_jf_j}+2(\log q)_i] w_i
\ge (\frac{3}{2}-\epsilon)q(x) w^2-C-Cw - \frac{1}{q(x)}.
\]
Next, fix again $x_*$ and we again take the cutoff function $\phi=\psi(|x-x_*|)$ as before. Then,
\[
\tilde{\cL}(\phi w)=\phi \tilde{\cL}w+w\tilde{\cL}\phi+2 \phi_i \phi^{-1}(w\phi)_i
-2 \phi_i^2\phi^{-1}w.
\]
By \eqref{auxeq:transcoefficient}, it is easy to know that
\[
w\tilde{\cL}\phi\ge -w(C+\epsilon |\phi_i|^2 u)\ge -\epsilon \phi q(x)w^2-Cw.
\]
Hence, one has
\[
\cL (\phi w)+\left[-4\frac{A(1-f)f_i+f_k f_{ki}}{A(1-f)^2-f_jf_j}+2(\log q)_i \right] (\phi w)_i
\ge q(x) \phi w^2-Cw-C - \frac{1}{q(x)}.
\]

Then, applying the maximum principle and repeating the same argument as for $v$,  one obtains
\[
w(t,x_*)\le \max\left\{\sup_{x\in B(x_*, 2(|x_*|+1))}\|w(0, x)\|, \quad C\right\}.
\]
Hence, $w$ is uniformly bounded and
\[
\|\nabla^2 f\|\le \sqrt{Aq}(1-f+\log M)
\]
according to \eqref{eq:aux1}. Applying Proposition \ref{pro:density-SDE}, the claim for the Hessian is then proved.
\end{proof}

\section{Ergodicity of the random splitting LMC}\label{sec:ergodicity}

For the convenience of the discussion in this section, we will abandon the time interpolation in \eqref{eq:timeinterpolation}, and regard the advection as a step that takes no time; that is, the diffusion is performed in $(t_n^+, t_{n+1}^-)$ and the advection is performed at $\{t_n\}=:[t_{n}^-, t_n^+]$.  Then the random splitting LMC is then a random pulse-diffusion model.
Throughout this section, let $(\xi_k)_{k\ge1}$ be the i.i.d. random splitting choices
and $(W_t)_{t\ge0}$ the driving Brownian motion.
We let the process $X(t)$ to experience a diffusion between $t_n^+<t<t_{n+1}^-$:
\[
dX=\sqrt{2\beta^{-1}}\,dW.
\]
If the diffusion happens first in the $n$th step, then
\[
X(t_n^+)=X_n, \quad X_{n+1}=S(X(t_{n+1}^-), \tau).
\]
Otherwise
\[
X(t_n^+)=S(X_n, \tau), \quad X_{n+1}=X(t_{n+1}^-).
\]

Consider a coupling copy of $X$, which we denote by $Y$.
Let $\theta$ be the stopping time that $X=Y$. Clearly, if $X\neq Y$, then $S(X, \tau)\neq S(Y, \tau)$ by the classical ODE
theory. Hence, $\theta$ can only be in $(t_n^+, t_{n+1}^-)$ for some $n$.
In the diffusion step, we let the dynamics of $Y$ to be
\begin{gather}
dY=
(I-2 \hat{e}\otimes \hat{e})dW_{t\wedge \theta}, \quad \hat{e}=\frac{X-Y}{|X-Y|}, \quad t<\theta,
\end{gather}
and
\begin{gather}
Y_t=X_t,\quad t\ge \theta.
\end{gather}
The dynamics of $Y$ in the diffusion step is well-defined, as one can define the dynamics with stopping times $\theta_j$ suitably
$\theta_j=\inf_{t>t_n}\{|X(t)-Y(t)|=1/j\}$ and then take $j\to\infty$ to get the limit dynamics for the dynamics of $Y$ (see \cite{eberle2016reflection,li2024geometric}).

 We define the discrete-time
information $\sigma$-algebra at step $n$ by
\[
\mathcal{G}_n := \sigma\!\bigl(X_0, Y_0\; (\xi_k)_{k\le n},\; (W_s)_{s\le t_n}\bigr),
\qquad n\ge0,
\]
which collects all randomness available at time \(t_n^+\).
Consequently, Brownian increments on \((t_n^+,t]\) are independent of \(\mathcal{G}_n\).
We also write 
\[Z_t:=X_t-Y_t\] be the difference of the two processes,  \(Z_n:=Z(t_n)\), and \(Z_n^\pm:=Z(t_n^\pm)\).

Then, one has
\[
dZ=2\sqrt{2\beta^{-1}}d\zeta_t,\quad 
\zeta_t=\int_{t_n\wedge\theta}^{t\wedge\theta}\frac{Z_s^{\otimes 2}}{|Z_s|^2}dW.
\]

As in  \cite{eberle2016reflection,eberle2019couplings,li2024geometric}.
We consider the function
\begin{gather}
f(r)=\int_0^r e^{-c_f (s\wedge R_1)}\,ds,
\quad R_1=3R.
\end{gather}

The main result in this section is the following.
\begin{theorem}\label{thm:ergodicity0}
There exists some constants $A_0>0$ and $\tau_0>0$. When $c_f> A_0\beta LR_1$ and $\tau\le \tau_0$, there is some constant $\lambda>0$ independent of $\tau$ such that
\begin{gather}
\E f(|Z_{n+1}|)\le e^{-\lambda \tau}\E f(|Z_n|).
\end{gather}
Consequently, let $\mu_n$ and $\nu_n$ be the laws of the random splitting LMC with initial laws $\mu_0$ and $\nu_0$, respectively. Then, it holds that
\[
W_1(\mu_n, \nu_n) \leq C e^{-\lambda n \tau} W_1(\mu_0, \nu_0),
\]
for some constant $C>0$ independent of $\tau$ and $n$.
\end{theorem}

To prove this theorem, we need some auxiliary lemmas. The first is the following.
\begin{lemma}\label{lem:gausstile}
It holds that 
\[
\P(|\zeta_t-\zeta_{s}|\ge u)\le 2\exp\left(-\frac{C\beta u^2}{\tau}\right).
\]
\end{lemma}
This essentially is the sub-Gaussian tail bound for diffusion process, which can be derived by computing the moments using BDG inequality. See \cite{li2024geometric} for the proof.

\begin{lemma}
For the diffusion step,  when $t_n^+<t < t_{n+1}^-$, one has
\[
\frac{d}{dt}\E f(|Z(t)|)=4\E \beta^{-1}f''(|Z|)1_{t\le \theta}
=-4\beta^{-1}c_f \E e^{-c_f |Z|\wedge R_1}1_{|Z|\le R_1}1_{t\le \theta}.
\]
\end{lemma}

This is a direct proof by It\^o's formula, by approximating $\theta$ using a sequence of stopping times $\theta_j$. Direct computation gives
\[
\begin{split}
&\nabla f(|x|)=f'(|x|)\frac{x}{|x|},\\
&\nabla^2 f(|x|)=f''(|x|)\frac{x\otimes x}{|x|^2}+f'(|x|)\frac{1}{|x|}(I-\frac{x\otimes x}{|x|^2}).
\end{split}
\]
Since $d\zeta_t$ contains only the martingale term while  the differential of the quadratic variation $d[\zeta_t, \zeta_t]$ is parallel
to $x$, one can obtain the result. We skip the details and refer the readers to \cite{eberle2011reflection,eberle2019quantitative,bou2020coupling} and \cite{li2024geometric} for details.

The following lemma is similar to \cite[Lemma 3.4]{li2024geometric}.
\begin{lemma}\label{lmm:rareeventestimate}
Let $0<a<b\le R_1$ and $\beta C a^2/\tau>4\log 8$.
Suppose that $G\in \mathcal{G}_n$ is an event such that $|Z_n|\le a$ on $G$.
Define $F=\{\exists s\in [t_n^+, t]: |Z_s|= b\}$. Then, for all $t\in [t_n^+, t_{n+1}^-]$,
\[
\E\!\big[ \mathbf{1}_G \mathbf{1}_{F} \mathbf{1}_{\{t<\theta\}}\big]
\,\le\, \eta_1(a,b,\tau)\, \E\!\big[ \mathbf{1}_G  \mathbf{1}_{\{t<\theta\}}\big],
\quad
\eta_1=4\exp\!\Big(-\tfrac{C\beta (b-a)^2}{\tau} \Big).
\]
Moreover,
\[
\E\!\big[ \mathbf{1}_G |Z_t| \mathbf{1}_{\{|Z_t|\ge b\}}\big]
\,\le\, \eta_2(a,b,\tau)\, \E\!\big[ \mathbf{1}_G |Z_t|\big],
\]
with
\[
\eta_2=\frac{6}{a}\left[b+\frac{\tau}{C\beta(b-a)}\right]
\exp\!\Big(-\tfrac{C\beta(b-a)^2}{2\tau}\Big).
\]
\end{lemma}

\begin{proof}
Define the event
\[
E=\{\exists s\in [t_n^+, t], |Z_s|=a\}\cap G, \quad B=\{t<\theta\}.
\]

(i). We consider the first claim.  Clearly, the event $F$ must be contained in $E$. We will actually show that
\[
\E[1_{G\cap E} 1_F  1_{t<\theta}]\le  \eta_1(a, b, \tau) \E[ 1_{G\cap E}   1_{t<\theta}].
\]
The idea is that once $E$ occurs, the probability of coupling before time $t$ is exponentially small in $a^2/\tau$, so $\{t<\theta\}$ holds with large probability. In fact,
\[
\P(G\cap E, F, t<\theta)=\P(G\cap E, t<\theta)\P(F | G\cap E, t<\theta).
\]
For the latter,
\[
\P(F | G\cap E, t<\theta)\le \frac{\P(F, G\cap E) }{\P(G\cap E)-\P(G\cap E, \theta\le t)}.
\]
Meanwhile,
\[
\P(G, E, t\ge \theta)=\P(G, E)\int_{t_n}^t \P(t\ge \theta| |Z_s|=a, G)\nu_{G,E}(ds),
\]
where $\nu_{G,E}(\cdot)$ is the conditional law for the first hitting time of $a$ for $|Z|$ with $\int_{t_n}^t \nu_{G,E}(ds)=1$.
By Lemma \ref{lem:gausstile}, we obtain
\[
\P(t\ge \theta| |Z_s|=a,G)\le \P(\sup_{s\le t' \le t\wedge\theta}2\sqrt{2\beta^{-1}}|\zeta_{t'}-\zeta_s|
\ge a | |Z_s|=a, G)\le 2\exp(-\frac{C\beta a^2}{\tau})<\frac{2}{8^4}.
\]
Hence, one has
\[
\P(F | G\cap E, t<\theta) \le 2\P(F | G\cap E)
\le 4\exp(-\frac{C\beta (b-a)^2}{\eta})=:\eta_1.
\]

(ii) The proof of the second claim uses similar ideas but more involved.
Note that $\{|Z_t|>b\}$ must be contained in $E$. We will then actually show that
\[
\E[ 1_{G\cap E} |Z_t| 1_{|Z_t|\ge b}]\le  \eta_2(a, b, \tau) \E[ 1_{G\cap E} |Z_t|].
\]
Clearly, one has the following two expressions.
\begin{gather*}
\begin{split}
 &\E[ 1_{G\cap E} |Z_t| 1_{|Z_t|\ge b} ]
=b \P(|Z_t|\ge b, G, E)+\int_b^{\infty} \P(|Z_t|\ge r, G, E)dr,\\
& \E[ 1_{G\cap E} |Z_t|  ]
=\int_0^{\infty} \P(|Z_t|\ge r, G, E)dr.
\end{split}
\end{gather*}
Hence, we need to show that
\[
b \P(|Z_t|\ge b | G, E)+\int_b^{\infty} \P(|Z_t|\ge r | G, E)dr
\le \eta_2\int_0^{\infty} \P(|Z_t|\ge r | G, E)dr.
\]

Intuitively, $\P(|Z_t|\ge r| G, E)$ is small for $r\ge b$ and almost $1$ if $r\le a/2$. 
In fact
\[
\P(|Z_t|\ge r| G, E)=\int_{t_n}^t \P(|Z_t|\ge r | |Z_s|=a, G, E)\nu_{G, E}(ds).
\]
For $r\ge b$, one has
\[
\P(|Z_t|\ge r| E,G) \le \sup_s \P(\sup_{s\le t' \le t\wedge\theta}2\sqrt{2\beta^{-1}}|\zeta_{t'}-\zeta_s|
\ge r-a | |Z_s|=a, G)\le 2\exp(-C\beta(r-a)^2/\tau).
\]
Similarly, for $r\le a/2$, one has $\P(|Z_t|>r | G, E, 1_{t<\theta})=1-\P(|Z_t|\le r | G, E, 1_{t<\theta})$. Similar as above,
\[
\P(|Z_t|\le r | G, E, 1_{t<\theta})\le \frac{1}{1-2\exp(-C\beta a^2/4\tau)}\P(|Z_t|\le r| G, E),
\]
while
\[
\P(|Z_t|\le r| G, E) \le  2\exp(-C\beta a^2/(4\tau)).
\]
Thus, with high probability the process remains above $a/2$, which provides a positive lower bound for the denominator.
With all the estimates here, it suffices the following holds
\[
\left[2b+\frac{2\tau}{C\beta(b-a)}\right]\exp\left(-\frac{C\beta}{2\tau}(b-a)^2\right)
\le \eta_2(a, b, \tau) \frac{a}{2}\frac{1-4\exp(-C\beta a^2/(4\tau))}{1-2\exp(-C\beta a^2/(4\tau))}.
\]
Hence, taking
\[
\eta_2(a, b,\tau)=\frac{2}{a}\frac{1-2/8}{1-4/8}\left[2b+\frac{2\tau}{C\beta(b-a)}\right]\exp\left(-\frac{C\beta}{2\tau}(b-a)^2\right)
\]
suffices.
\end{proof}

We also have the following.
\begin{lemma}\label{lmm:largeestimate}
Let $ a>0$ satisfy $2\exp(-C\beta a^2/\tau)<1$. Suppose that $G\in \mathcal{G}_n$  is an event such that  $2a< |Z_n|\le  b$ for some $b>0$ on $G$. Define $F=\{\exists s\in [t_n^+, t], |Z_s-Z_n^+|= a\}$. Then, one has
\begin{gather}
\E[ 1_G 1_{F} f(|Z_t|)]\le  \eta_3(2a, b, \tau) \E[ 1_G  f(|Z_t|)],
\end{gather}
where
\[
\eta_3(2a, b,\tau)=\frac{2(2b+\eta_2(b,2b,\tau)) \exp(-C\beta a^2/\tau) }{f(a)(1-2\exp(-C\beta a^2/\tau))}.
\]
\end{lemma}
The proof of Lemma \ref{lmm:largeestimate} is almost the same as the proof of Lemma \ref{lmm:rareeventestimate}. The estimates needed include: (1). The left hand side is controlled by
\[
\E[ 1_G 1_{F} |Z_t|]\le [2b+\eta_2(b, 2b, \tau)]\P(G, F)\le 2[2b+\eta_2(b, 2b, \tau)]\exp(-C\beta a^2/\tau)\P(G).
\]
 (2). The expectation on the right hand side is bounded below by 
 \[
 \E[ 1_G  f(|Z_t|)]\ge f(a)\P(G)[1-\P(|Z_t|\le a| G)]\ge f(a)\P(G)(1-2\exp(-C\beta a^2/\tau)).
 \]
  We omit the details for the proof.  Next, we give the proof of the main theorem.
\begin{proof}[Proof of Theorem \ref{thm:ergodicity0}]
Recall that $R_1=3R$. We decompose the whole probability space into three parts.
\[
\Omega_1=\{ |Z_n| < \tau^{1/2-\delta}\},\quad
\Omega_2=\{\tau^{1/2-\delta}\le |Z_n| < 2R\},
\quad \Omega_3=\{|Z_n|\ge 2R\}.
\]

We discuss the three cases separately. In $\Omega_3$, the advection step will contribute the contraction due to the convexity. When $|Z_n|\le 2R$, the main idea is that the reflection coupling together with the concavity of $f$ gives contraction with coefficient $c_f$, which is large to dominate the advection step.

\noindent {\bf Case 1}

We consider $\Omega_1$. The main difficulty is that the probability of $1_{t\le \theta}$ may not be close to $1$.
The observation is that $f(|Z|)1_{t\le\theta}=f(|Z|)$ and $|Z|1_{t\le \theta}=|Z|$. Moreover, $e^{-c_f |Z|\wedge R_1}$ is close to $1$
so that $|Z|$ and $f(|Z|)$ are almost the same.

\noindent{\bf Case 1.1}
If the diffusion step is first, then $X(t_n^+)=X(t_n)$ and $Y(t_n^+)=Y(t_n)$. We denote
\[
F_1(t)=\{\exists s\in [t_n, t]: |Z_s|=2\tau^{1/2-\delta}\}.
\]
On the event $F_1(t)^c$, $|Z_t|\le R_1$ clearly holds. Then, since $\exp(-2c_f \tau^{1/2-\delta})\ge 3/4$, one has
\begin{gather*}
\frac{d}{dt}\E 1_{\Omega_1}f(|Z_t|)\le -3\beta^{-1}c_f \E 1_{\Omega_1\cap F_1(t)^c} 1_{t<\theta}.
\end{gather*}
By the first claim in Lemma \ref{lmm:rareeventestimate}, one has
\[
\E 1_{\Omega_1\cap F_1(t)^c} 1_{t<\theta}\ge (1-\eta_1(\tau^{1/2-\delta}, 2\tau^{1/2-\delta},\tau))\E 1_{\Omega_1} 1_{t<\theta}.
\]
Hence, for $\tau$ small enough, one has
\begin{gather*}
\frac{d}{dt}\E 1_{\Omega_1}f(|Z_t|)\le -2\beta^{-1}c_f \E 1_{\Omega_1} 1_{t<\theta}
\le -\beta^{-1}c_f \E 1_{\Omega_1} 1_{t<\theta}-\beta^{-1}c_f \E 1_{\Omega_1} 1_{t_{n+1}<\theta}.
\end{gather*}
We will use the latter to control the contribution from the subsequent advection step.
In fact, for the advection step, letting $\bar{X}_s:=S(X_{n+1}^-, s)$, $\bar{Y}_s:=S(Y_{n+1}^-, s)$ and $W_s=\bar{X}_s-\bar{Y}_s$, one has
\[
\E 1_{\Omega_1}f(|Z_{n+1}|)-\E 1_{\Omega_1}f(|Z_{n+1}^-|)\le L \int_{0}^{\tau}\E 1_{\Omega_1}1_{\theta>t_{n+1}}f'(|W_s|)|W_s|\le C\tau L \E 1_{\Omega_1}|Z_{n+1}^-|,
\]
where we have used $1_{t_{n+1}<\theta}|Z_{n+1}^-|=|Z_{n+1}^-|$. By the second claim in Lemma \ref{lmm:rareeventestimate}, one has
\begin{multline*}
\E 1_{\Omega_1}|Z_{n+1}^-|
\le (1+\eta_2(\tau^{1/2-\delta}, 2\tau^{1/2-\delta}, \tau))\E 1_{\Omega_1}|Z_{n+1}^-| 1_{|Z_{n+1}^-|\le 2\tau^{1/2-\delta}}\\
\le 2\tau^{1/2-\delta}(1+\eta_2(\tau^{1/2-\delta}, 2\tau^{1/2-\delta}, \tau))\E 1_{\theta>t_{n+1}}1_{\Omega_1}.
\end{multline*}
The number $\eta_2$ is small if $\tau$ is small.  Hence, one can choose $c_f$ large such that
\begin{gather*}
\E 1_{\Omega_1}f(|Z_{n+1}|)
\le \E 1_{\Omega_1}f(|Z_{n}|)-\beta^{-1}c_f \int_{t_n^+}^{t_{n+1}^-}\E 1_{\Omega_1} 1_{t<\theta}.
\end{gather*}

Let $u(t):= \E 1_{\Omega_1}f(|Z_{s}|)$. By the estimate for the diffusion step, for $t<t_{n+1}$, one has 
\begin{multline*}
u(t)\le u(t_n)-\beta^{-1}c_f \int_{t_n^+}^{t}\E 1_{\Omega_1} 1_{t<\theta}
\le u(t_{n})-\frac{\beta^{-1}c_f}{2\tau^{1/2-\delta}} \int_{t_n}^{t_{n+1}^-}\E 1_{\Omega_1} |Z_t|1_{|Z_t|\le 2\tau^{1/2-\delta}}\\
\le u(t_{n})-\frac{\beta^{-1}c_f}{2\tau^{1/2-\delta}}(1-\eta_2)\int_{t_n}^t u(s)\,ds,
\end{multline*}
where have used the fact $|Z_t|1_{t<\theta}=|Z_t|$ here.

To apply the Gr\"onwall's lemma, the problem is that
$\E 1_{\Omega_1}f(|Z_{n+1}|)\neq u(t_{n+1}^-)$. To resolve this, define $b(t):=\max\{\E 1_{\Omega_1}f(|Z_{n+1}|), u(t)\}$.
Since $\E 1_{\Omega_1}f(|Z_{n+1}|) \le u(t_n)$ by the estimate above, one concludes that $b$ is continuous and $b(t_n)=u(t_n)$.
It is clear that $u(t)$ is monotonically decreasing and $\E 1_{\Omega_1}f(|Z_{n+1}|) \le Cu(t_{n+1}^-)$ since the advection step amplify $|Z|$
by a small factor. Hence, one has $u(s)\ge C^{-1} b(s)$ for some constant. Then,
\[
b(t)\le b(t_n)-c\int_{t_n}^t b(s)\,ds.
\]
By applying Gr\"onwall's lemma to the function $b(t)$, we conclude that 
\[
\E 1_{\Omega_1} f(|Z_{n+1}|)\le e^{-c\tau}\E 1_{\Omega_1} f(|Z_n|).
\]
{\bf Case 1.2} If the advection step is first, after the advection, $|Z_n^+|\le e^{CL\tau}|Z_n|$.
Moreover, $f'(|Z|)\ge 3/4$. Hence, $u(t_n^+)\le e^{C_1L\tau}u(t_n)$.
The diffusion step is estimated similarly as above, and one has
\[
u(t)\le u(t_n^+)-c\int_{t_n}^t u(s)\,ds
\]
where $c=\frac{\beta^{-1}c_f}{2\tau^{1/2-\delta}}(1-\eta_2)$ which is greater than $2C_1L$ if $\tau$ is small. 
Combining the two cases here, one gets
\[
\E 1_{\Omega_1} f(|Z_{n+1}|)\le e^{-c'\tau}\E 1_{\Omega_1}f(|Z_n|).
\]

\noindent{\bf Case 2}  On this event, the good thing is that $\theta\le t$ will almost not happen. However, $e^{-c_f |Z_t|}$ is no longer close to $1$. Since $c_f$ is large, we cannot 
bound this by $e^{-c_f 2R}$ from below here. To address this, we define
$\mu_m:=m\tau^{1/2-\delta}$
and decompose $\Omega_2$ into the union of the following sets
\[
\Omega_{2,m}:=\{ \mu_m \le |Z_n|< \mu_{m+1}\}, \quad m=1, 2, \cdots, \lceil 2R/\tau^{1/2-\delta} \rceil-1.
\]

During the diffusion step, one has 
\[
\frac{d}{dt}\E 1_{\Omega_{2,m}}f(|Z(t)|)=-4\beta^{-1}c_f \E 1_{\Omega_{2,m}}e^{-c_f |Z_t|}1_{|Z_t|\le R_1}1_{t\le \theta}.
\]
Using the fact that $e^{-CL\tau}\mu_m \le |Z_n| \le \mu_{m+1}e^{CL\tau}
\le 2Re^{CL\tau}$  (if there is no advection step before the diffusion, the constants here are still correct), one has 
\begin{multline*}
 \E [1_{\Omega_{2,m}}e^{-c_f |Z_t|}1_{|Z_t|\le R_1}1_{t\le \theta}]
 \ge e^{-c_f \mu_{m+1}e^{CL\tau}}
 \E[ 1_{\Omega_{2,m}}e^{-c_f |Z_t-Z_n^+|} 1_{|Z_t|\le R_1}1_{t\le \theta} ]\\
 \ge \frac{1}{2}\frac{e^{-c_f \mu_m}}{f(R_1)}\E [1_{\Omega_{2,m}}1_{|Z_t-Z_n^+|< a} f(|Z_t|)].
\end{multline*}
Clearly, if $\tau$ is small enough, applying Lemma \ref{lmm:largeestimate} for $a=3\tau^{1/2-\delta}/8$, one has
\begin{gather*}
\E 1_{\Omega_{2,m}}e^{-c_f |Z_t|}1_{|Z_t|\le R_1}1_{t\le \theta} 
\ge \frac{1}{4 f(R_1)} e^{-c_f \mu_m}\E [1_{\Omega_{2,m}} f(|Z_t|)].
\end{gather*}

If the diffusion step is first, we define $W_s$ similarly as in Case 1. Then one has
\[
\E 1_{\Omega_{2,m}}f(|Z_{n+1}|)-\E 1_{\Omega_{2,m}}f(|Z_{n+1}^-|)
\le L\int_0^{\tau}\E e^{-c_f|W_s|\wedge R_1}1_{t_{n+1}<\theta}|W_s|1_{\Omega_{2,m}}
\]
It is clear that $\partial_s|W_s-Z_{n+1}^-|
\le L|W_s|$ so that $|W_s-Z_{n+1}^-|\le |Z_{n+1}^-|(e^{L\tau}-1)$ and $|W_s|\le e^{L\tau}|Z_{n+1}^-|$. Note that $|Z_{n+1}^-|1_{t_{n+1}<\theta}=|Z_{n+1}^-|$. Then, one has for $\tau\le \tau_0$ (where $\tau_0$ depends on $c_f, R_1,L$ etc)
\begin{gather*}
\E e^{-c_f|W_s|\wedge R_1}1_{t_{n+1}<\theta}|W_s|1_{\Omega_{2,m}}\le 2\E e^{-c_f|Z_{n+1}^-|\wedge R_1}|Z_{n+1}^-|
1_{\Omega_{2,m}},
\end{gather*}
By the second claim in Lemma \ref{lmm:rareeventestimate}, 
\[
\E e^{-c_f|Z_{n+1}^-|\wedge R_1}|Z_{n+1}^-|
1_{\Omega_{2,m}}
\le (1+e^{-c_f R_1}\eta_2(\mu_{m+1}, R_1, \tau))\E e^{-c_f|Z_{n+1}^-|}|Z_{n+1}^-|1_{|Z_{n+1}^-|\le R_1}
1_{\Omega_{2,m}}.
\]
Since $r/f(r)$ is an increasing function on $(0,\infty)$, 
\[
\E e^{-c_f|Z_{n+1}^-|}|Z_{n+1}^-|1_{|Z_{n+1}^-|\le R_1}
1_{\Omega_{2,m}}\le \frac{R_1}{f(R_1)}\E e^{-c_f|Z_{n+1}^-|}f(|Z_{n+1}^-|)1_{|Z_{n+1}^-|\le R_1}
1_{\Omega_{2,m}}.
\]
Applying Lemma \ref{lmm:largeestimate} again,
\begin{gather*}
\begin{split}
&\E e^{-c_f|Z_{n+1}^-|}f(|Z_{n+1}^-|)1_{|Z_{n+1}^-|\le R_1}
1_{\Omega_{2,m}}\\
&\le \eta_3\E f(|Z_{n+1}^-|)1_{\Omega_{2,m}}
+\E[e^{-c_f|Z_{n+1}^-|}f(|Z_{n+1}^-|)1_{|Z_{n+1}^--Z_n|<a}1_{\Omega_{2,m}}]\\
&\le \eta_3\E f(|Z_{n+1}^-|)1_{\Omega_{2,m}}+2e^{-c_f \mu_m}\E[f(|Z_{n+1}^-|)1_{|Z_{n+1}^--Z_n|<a}1_{\Omega_{2,m}}]\\
& \le (\eta_3+2e^{-c_f \mu_m})\E f(|Z_{n+1}^-|)1_{\Omega_{2,m}}.
\end{split}
\end{gather*}
Clearly, when $\tau$ is small enough, we have $\eta_3\le e^{-c_f \mu_m}$ as $\eta_3$ is exponentially small.

Hence one has
\begin{multline*}
\E 1_{\Omega_{2,m}}f(|Z_{n+1}|)\le \left[1+A_0 L\frac{R_1}{f(R_1)}e^{-c_f \mu_m}\tau \right]\E [f(|Z_{n+1}^-|) 1_{\Omega_{2,m}}]\\
\le \left[1+A_0 L\frac{R_1}{f(R_1)}e^{-c_f \mu_m}\tau \right]\exp\left(-\beta^{-1}\frac{c_f}{f(R_1)}e^{-c_f \mu_m}\tau\right)\E 1_{\Omega_{2,m}}f(|Z_n|). 
\end{multline*}
Note that the constant $A_0$ here is independent of $c_f, R_1,L,\mu_m$.

If the advection step is before the diffusion step, obtaining such an estimate is similar and is slightly easier.

Here the key point is that the multiplicative factor 
is strictly less than $e^{-c\tau}$ for some $c>0$ independent of $\tau$, if $c_f> A_0\beta LR_1$. Hence, choosing such $c_f$, one gets the contraction to obtain
\[
\E 1_{\Omega_{2,m}} f(|Z_{n+1}|)\le e^{-c\tau}\E 1_{\Omega_{2,m}} f(|Z_n|).
\]

\noindent {\bf Case 3}

We consider the case for $\Omega_3$, which is much easier as the advection step will yield the contraction. If the advection is first , the contraction is clear in the advection step since
\[
Z_s\cdot (b(X_s)-b(Y_s))\le -\lambda |Z_s|^2,
\] 
while $\E 1_{\Omega_3}f(|Z_t|)$
is nonincreasing in the diffusion step, one easily has
\[
\E 1_{\Omega_3}f(|Z_{n+1}|)\le \E 1_{\Omega_3}f(|Z_{n+1}^-|)\le e^{-c\tau}\E f(|Z_n|).
\]

If the diffusion step is first, then $\E 1_{\Omega_3}f(|Z_t|)$ is nonincreasing in the diffusion step. In the advection step
\[
\E 1_{\Omega_3}f(|Z_{n+1}|)
=\E 1_{\Omega_3}f(|Z_{n+1}|) 1_{|Z_{n+1}^-|< \frac{3}{2}R}+\E 1_{\Omega_3}f(|Z_{n+1}|) 1_{|Z_{n+1}^-|\ge \frac{3}{2}R}.
\]
By the contraction due to the convexity, one has
\[
\E 1_{\Omega_3}f(|Z_{n+1}|) 1_{|Z_{n+1}^-|\ge \frac{3}{2}R}
\le e^{-c_1\tau} \E 1_{\Omega_3}f(|Z_{n+1}^-|) 1_{|Z_{n+1}^-|\ge \frac{3}{2}R}
\]
while
\[
\E 1_{\Omega_3}f(|Z_{n+1}|) 1_{|Z_{n+1}^-|< \frac{3}{2}R}\le \exp(C\tau)\E 1_{\Omega_3}f(|Z_{n+1}^-|) 1_{|Z_{n+1}^-|< \frac{3}{2}R}.
\]
Since $|Z_n|>2R$,  we apply Lemma \ref{lmm:largeestimate} again with $a=R/2$ and obtain
\[
\E 1_{\Omega_3}f(|Z_{n+1}^-|) 1_{|Z_{n+1}^-|< \frac{3}{2}R}
\le \eta_3 \E 1_{\Omega_3}f(|Z_{n+1}^-|).
\]
Hence
\[
\begin{split}
\E 1_{\Omega_3}f(|Z_{n+1}|) &\le (e^{-c_1\tau}+(e^{C\tau}-e^{-c_1\tau})\eta_3)\E 1_{\Omega_3}f(|Z_{n+1}^-|)\\
&\le (e^{-c_1\tau}+(e^{C\tau}-e^{-c_1\tau})\eta_3)\E 1_{\Omega_3}f(|Z_{n}|).
\end{split}
\]
When $\tau$ is small, $\eta_3$ is exponentially small so that
\[
e^{-c_1\tau}+(e^{C\tau}-e^{-c_1\tau})\eta_3\le e^{-c\tau}
\]
for some $c$ independent of $\tau$, and thus
\[
\E 1_{\Omega_3} f(|Z_{n+1}|)\le e^{-c\tau}\E 1_{\Omega_3} f(|Z_n|).
\]

Summing over all disjoint events $\Omega_1,\Omega_2,\Omega_3$, we obtain
\[
\E f(|Z_{n+1}|)=\sum_{i=1}^3 \E 1_{\Omega_i} f(|Z_{n+1}|)
\le e^{-c\tau}\sum_{i=1}^3 \E 1_{\Omega_i} f(|Z_n|) 
= e^{-c\tau}\E f(|Z_n|).
\]
This proves the main inequality in the theorem.

We take the coupling measure $\Gamma$ for $(X, Y)$
such that
\[
\E_\Gamma|X_0-Y_0|=W_1(\mu_0, \nu_0),
\]
and the coupling later is propagated by the reflection coupling 
of the Brownian motions along the random splitting dynamics.
Let $\mu_n^{\xi}$ and $\nu_n^{\xi}$ be the laws of $X$
and $Y$ respectively for the random splitting LMC with given sequence of $\xi$.
 By definition of the Wasserstein distance,
\[
W_1(\mu_n^{\xi},\nu_n^{\xi})\le \E_{\Gamma}|X_n-Y_n| = \E |Z_n|.
\]
Since $f(r)=e^{-c_f(r\wedge R_1)}r$ is equivalent to $r$ (there exist $c_1,c_2>0$ such that $c_1 r \le f(r)\le c_2 r$ for all $r$), we have
\[
\E |Z_n|\le C \E f(|Z_n|).
\]
Applying the contraction inequality for $\E f(|Z_n|)$ yields
\[
W_1(\mu_n^{\xi},\nu_n^{\xi})\le C e^{-\lambda n\tau} W_1(\mu_0,\nu_0).
\]
Since $\mu_n=\E_\xi \mu_n^\xi$ and $\nu_n=\E_\xi \nu_n^\xi$, 
the joint convexity of $W_1$ (which follows immediately from its 
Kantorovich--Rubinstein dual formulation, see Villani~\cite[Particular Case~5.16]{villani2008optimal}) yields
\[
W_1(\mu_n,\nu_n)\;\le\; \E_\xi\,W_1(\mu_n^\xi,\nu_n^\xi).
\]
\end{proof}

\section{Finite time error estimate}\label{sec:error-estimate}

This section presents a systematic error analysis for the random splitting LMC. Through the pointwise bounds of the gradient and Hessian for the logarithmic density, we establish precise control of the temporal evolution for the relative entropy of the density with respect to the exact solution of the Fokker-Planck equation. Our analysis yields a fourth-order relative entropy bounds with explicit $T$-dependence, which is then extended to uniform-in-time control via the ergodicity properties.

\subsection{Finite time error bound in relative entropy}
We will get back to the time interpolation in \eqref{eq:timeinterpolation}.
 Let us compute
\[
E(t) := \mathcal{H}(\bar{\rho}(t) \parallel \rho(t)),
\]
where $\rho(t)$ is the solution of the Fokker-Planck equation for the original SDE. To compute the derivative, one may need the evolution of $\rho^{\xi_n}(t)$ defined in \eqref{eq:timeinterpolation}.
We can compute that
\begin{multline}\label{eq:FP-numerical}
    \partial_t\rho^{\xi_n}=\cL_{\xi(1)} \rho^{\xi}+e^{(t-n\tau)\cL_{\xi(1)}}\cL_{\xi(2)} e^{(t-n\tau)\cL_{\xi(2)}}\bar{\rho}_n\\
=\cL \rho^{\xi_n}+[e^{(t-n\tau)\cL_{\xi(1)}}, \cL_{\xi(2)}]e^{(t-n\tau)\cL_{\xi(2)}}\bar{\rho}_n.
\end{multline}
Here, the commutator $[e^{(t-n\tau)\mathcal{L}_{\xi(1)}}, \mathcal{L}_{\xi(2)}]$ is defined as:
\[
[e^{(t-n\tau)\mathcal{L}_{\xi(1)}}, \mathcal{L}_{\xi(2)}] := e^{(t-n\tau)\mathcal{L}_{\xi(1)}} \mathcal{L}_{\xi(2)} - \mathcal{L}_{\xi(2)} e^{(t-n\tau)\mathcal{L}_{\xi(1)}}.
\]
This operator quantifies the non-commutativity between the evolution semigroup $e^{(t-n\tau)\mathcal{L}_{\xi(1)}}$ and the generator $\mathcal{L}_{\xi(2)}$ in our splitting scheme.
The remainder term is of order $O(\tau)$. If we consider the expectation, we have
\begin{equation}\label{eq:FP-numerical1}
    \partial_t\bar{\rho}(t)
=\cL \bar{\rho}+\E_{\xi_n}[e^{(t-n\tau)\cL_{\xi(1)}}, \cL_{\xi(2)}]e^{(t-n\tau)\cL_{2}}\bar{\rho}_n.
\end{equation}
We expect the remainder to reduce to $O(\tau^2)$ after expectation as the leading term in bias cancels. Of course, we need to justify this rigorously later to close the estimates.

Compute the formula for 
\[
\frac{d}{dt}\mathcal{H}(\bar{\rho}(t) \parallel \rho(t))
=\int (1+\log\bar{\rho}-\log\rho(t))\partial_t\bar{\rho}-\frac{\bar{\rho}}{\rho(t)}\partial_t\rho(t).
\]
Using \eqref{eq:SDE}, \eqref{eq:FP-numerical1} (recalling $b=-\nabla U$), and integration by parts, it is then reduced to 
\[
\beta^{-1} \int (-\nabla\log\frac{\bar{\rho}}{\rho}\cdot\nabla\bar{\rho}
+\bar{\rho}\nabla\log\frac{\bar{\rho}}{\rho}\cdot\nabla\log\rho)
+\int (\log\frac{\bar{\rho}}{\rho})\mathbb{E}_{\xi}[e^{(t-n\tau)\cL_{\xi(1)}}, \cL_{\xi(2)}]e^{(t-n\tau)\cL_{2}}\bar{\rho}_n.
\]
Simplifying this leads to
\begin{equation}\label{eq:relative}
    \frac{d}{dt}\mathcal{H}(\bar{\rho}(t) \parallel \rho(t))
    = -\beta^{-1}\int \bar{\rho}|\nabla\log\frac{\bar{\rho}}{\rho}|^2\,dx
+r_t,
\end{equation}
where
\[
 r_t=\int (\log\frac{\bar{\rho}}{\rho})\mathbb{E}_{\xi}[e^{(t-n\tau)\cL_{\xi(1)}}, \cL_{\xi(2)}]e^{(t-n\tau)\cL_{2}}\bar{\rho}_n.
\]

Next, our main goal is to estimate the remainder term. We first derive the formulas of the remainder term.
To proceed, we denote
\[
s:=t-n\tau.
\]
Then,
\begin{multline}
2 \mathbb{E}_{\xi}[e^{s\cL_{\xi(1)}}, \cL_{\xi(2)}] e^{s\cL_{\xi(2)}}\bar{\rho}_n
= ( e^{s\cL_{1}} \cL_{2} - \cL_{2} e^{s\cL_{1}}) h_2(s, x) + ( e^{s\cL_{2}} \cL_{1} - \cL_{1} e^{s\cL_{2}}) h_1(s, x)\\
= : I_1 + I_2,
\end{multline}
where we have introduced
\begin{gather}
\begin{split}
& h_1(s, x)=e^{s\cL_{1}}\bar{\rho}_n(x)=\bar{\rho}_n(\phi_{-s}(x))J(-s, x),\\
& h_2(s, x)=e^{s\cL_{2}}\bar{\rho}_n(x)
=\int_{\R^d} \frac{1}{(4\pi\beta^{-1}s)^{d/2}}\exp(-\frac{\beta|x-y|^2}{4s})\bar{\rho}_n(y)\,dy.
\end{split}
\end{gather}

It follows that
\begin{multline}
I_1 =( e^{s\cL_{1}} \cL_{2} - \cL_{2} e^{s\cL_{1}}) h_2(s, x) \\
=\beta^{-1} (\Delta h_2)(s, \phi_{-s}(x)) J(-s,x)
-\beta^{-1}  \Delta_x \big( h_2(s, \phi_{-s}(x))J(-s,x) \big),
\end{multline}
and
\begin{multline}
I_2  =( e^{s\cL_{2}} \cL_{1} - \cL_{\xi(1)} e^{s\cL_{2}}) h_1(s, x) =
- \int \frac{1}{(4 \pi \beta^{-1} s)^{\frac{d}{2}}}
e^{-\frac{\beta |x-y|^2}{4 s} } \nabla_y\cdot ( b (y) h_1(s, y)) dy \\
		+\nabla_x \cdot\Big(  b(x) \int  \frac{1}{(4 \pi \beta^{-1} s)^{\frac{d}{2}}}
		e^{-\frac{\beta|x-y|^2}{4  s} } h_1(s, y) dy \Big).
\end{multline}

After some computation whose details are provided in Appendix \ref{app:pointremainder},  we obtain (recall $s=t-n\tau$)
\begin{gather}\label{eq:remainderformula}
\begin{split}
r_{t}&=\frac{1}{2}\int_{\R^d} (\log\frac{\bar{\rho}}{\rho}) ( I_1 + I_2)\,dx
=\frac{\beta^{-1}s}{2}\int \nabla \log\frac{\bar{\rho}}{\rho}\cdot (\sum_{i=1}^3E_i(s, x))\,dx\\
&\le \frac{\beta^{-1}}{8}\int |\nabla \log\frac{\bar{\rho}}{\rho}|^2\bar{\rho}\,dx
+\frac{3(\beta^{-1}s)^2}{2}\int \frac{\sum_{i=1}^3|E_i(s,x)|^2}{\bar{\rho}(t,x)}\,dx,
\end{split}
\end{gather}
where the $E_i$ are three remainder vectors fields. The remainder $E_1$  is given by
\begin{multline}
E_1(s,x)=\frac{\nabla_x\phi_{-s}(x)-(\nabla_x\phi_{-s}(x))^{-T}}{s}\cdot\nabla h_2(s, \phi_{-s} (x) ) J(-s,x )\\
+\int \frac{e^{-\frac{\beta |x-y|^2}{4 s} }}{(4 \pi \beta^{-1} s)^{\frac{d}{2}}}  (\overline{\nabla b}(x, y)+\overline{\nabla b}(x, y)^T) \cdot \nabla  h_1(s, y) dy,	
\end{multline}
where we have introduced
\begin{gather}
\overline{\nabla b}(x, y):=\int_0^1\nabla b(\lambda y+(1-\lambda)x)d\lambda.
\end{gather}
Moreover, 
$E_2$ is given by
\begin{multline}
E_2(s,x)=s^{-1}\nabla J(-s, x) h_2(s, \phi_{-s}(x))+
\int \frac{e^{-\frac{\beta |x-y|^2}{4 s} }}{(4 \pi \beta^{-1} s)^{\frac{d}{2}}}  \overline{\nabla (\nabla\cdot b)}(x,y)    h_1(s, y) dy,
\end{multline}
and $E_3$ is given by
\begin{multline}
E_3(s,x)=\int \frac{e^{-\frac{\beta |x-y|^2}{4 s} }}{(4 \pi \beta^{-1} s)^{\frac{d}{2}}}  \int_0^1 \lambda [\Delta b(z(\lambda,y,x))-\Delta b(z(\lambda, x, y))]d\lambda    h_1(s, y) dy,
\end{multline}
where $z(\lambda, x, y):=\lambda x+(1-\lambda)y$.

By the definition of $h_i$, one has
\begin{gather}
\nabla h_2(s, x)=\int \frac{1}{(4\pi \beta^{-1}s)^{d/2}}\exp(-\frac{\beta|x-y|^2}{4s})\nabla \bar{\rho}_n(y)\,dy,
\end{gather}
and
\begin{gather}
\nabla h_1(s,x)=\nabla \phi_{-s}(x)\cdot \nabla\bar{\rho}_n(\phi_{-s}(x)) J(-s, x)
+\bar{\rho}_n(\phi_{-s}(x))\nabla J(-s, x).
\end{gather}

We expect that all the three terms above are of order $O(s)$ so that $r_t$ is of order $O(s^4)$.
In fact, we are able to derive the following pointwise bounds using the pointwise estimates of the log density and its various derivatives in section \ref{subsec:densitybound} and section \ref{subsec:derivativebound}.
The derivation is kind of tedious and we attach it in Appendix \ref{app:pointremainder}.
\begin{proposition}\label{pro:deri_estimate}
Suppose Assumption \ref{ass:drift} on $b$ and Assumption \ref{ass:initial-density} on the initial density hold, then 
\begin{gather}
\sum_{i=1}^3|E_i(s, x)|\le Cs(1+|x|^{p_1})\bar{\rho}(t, x),
\end{gather}
for some $p_1>0$, where $C$ depends on $T$, and possibly the dimension $d$.
\end{proposition}

\begin{remark}
What is crucial to us is that the bound here depends on the coefficients $ A, c , \delta$ in the assumption of $\rho_0$, not on the particular form of $\rho_0$. This is important as we will apply this result to different "initial densities" and the contraction property to get the global control.
\end{remark}
Combining equation \eqref{eq:relative}, \eqref{eq:remainderformula}, and Proposition \ref{pro:deri_estimate} and the moment bounds in Lemma \ref{lem:moment_hybrid_refined}, we obtain the following result.
\begin{theorem}
	\label{thm:uniform_error}
Suppose  Assumption \ref{ass:drift} holds. Let $\rho(t)$ be the solution to the Fokker-Planck equation with initial value $\rho_0$ satisfying Assumption   \ref{ass:initial-density} and $\bar{\rho}(t)$ be the law to the random splitting LMC with the same initial value $\rho_0$. Then, for any $T>0$, there is a constant $C(T)>0$ such that 
		\begin{equation}
			\mathcal{H}(\bar{\rho}(t) \| \rho(t)) \leq C(T) \tau^4.
			\label{eq:local_error}
		\end{equation}
\end{theorem}
By Assumption~\ref{ass:drift} and the exponential moment bounds in 
Assumption~\ref{ass:initial-density}, the reference law $\rho(t)$ satisfies, 
uniformly for $t\in[0,T]$, the Talagrand--$T_1$ transportation inequality
\[
W_1(\mu,\rho(t)) \;\le\; \sqrt{\,2C_T\,\mathcal{H}(\mu\|\rho(t))}, 
\qquad \text{for all probability measures $\mu$}.
\]

\begin{corollary}[Wasserstein--1 error bound]\label{cor:W1_from_H}
Under the assumptions of Theorem~\ref{thm:uniform_error}, 
for any $T>0$ there exists a constant $C(T)>0$ such that, for all $t\in[0,T]$,
\[
W_1\bigl(\bar\rho(t),\rho(t)\bigr)
\;\le\; \sqrt{\,2C_T\,\mathcal{H}(\bar\rho(t)\|\rho(t))\,}
\;\le\; C(T) \,\tau^{2}.
\]
\end{corollary}

In the next several subsections, we establish the bounds for the log-density and its derivatives for $\bar{\rho}(t)$, which serve as fundamental priori estimates for subsequent error analysis. Our approach is to get the estimates for $\rho_n^{\xi}$ defined in \eqref{eq:rhoxi} first in section \ref{subsec:densitybound}--section \ref{subsec:derivativebound}. Then, we extend this estimate to $\bar{\rho}_n$ and $\bar{\rho}(t)$ in section \ref{subsec:boundsforrhobar}.

\subsection{The bounds on the density}\label{subsec:densitybound}
This subsection establishes rigorous upper and lower bounds for the numerical solution density.

\begin{proposition}\label{pro:densityestimate}
Suppose Assumptions \ref{ass:drift} and \ref{ass:initial-density} hold. Then for any fixed sequence of random permutations $\b{\xi}$ and $\tau\le 1$, the density for the random splitting LMC $\rho_n^{\xi}$ defined in \eqref{eq:rhoxi} satisfies
\[
A_3\exp(-\alpha_1 |x|^p)\le \rho_n^{\xi} \le A_4\exp(-\alpha_2 |x|^2),
\quad \forall~ n\tau\le T,
\]
where the constants $\alpha_i$ and $A_3, A_4$ are independent of $\b{\xi}$ and $\tau$, but can depend on $T$ and the initial data through the constants in Assumption \ref{ass:initial-density}.
\end{proposition}

To prove this result, we first establish some simple but useful properties of the advection step and the diffusion step. The following is about the short time continuity of the transport flow map.
\begin{lemma}\label{lem:flowmap}
Under Assumption \ref{ass:drift} and
\ref{ass:initial-density}, the flow map $\phi_t$ associated with SDE \eqref{eq:SDE} satisfies the following properties:
\[
|\phi_t(x)-x|\le C(1+|x|)(e^{L|t|}-1),
\]
\[
|x|e^{-Lt}-C(1-e^{-Lt})  \le |\phi_{-t}(x)|\le |x|e^{Lt}+C(e^{Lt}-1),
\]
and for any two points
\begin{gather}
e^{-L t }|\phi_{-t}(x)-y|  \le |x-\phi_t(y)| \le e^{Lt}|\phi_{-t}(x)-y|.
\end{gather}
\end{lemma}

\begin{proof}
For the first, one notes that 
\[
\partial_t|\phi_t(x)-x|^2=2(\phi_t(x)-x)\cdot b(\phi_t(x)) \Longrightarrow  
\partial_t |\phi_t(x)-x| \le |b(\phi_t(x))|\le  L|\phi_t(x)-x|+|b(x)|.
\]
 The first claim for $t\ge 0$ follows from the Gr\"onwall's inequality and linear growth of $b(x)$.
 For $t<0$, one has similarly
 \[
 \partial_t |\phi_t(x)-x| \ge   -L|\phi_t(x)-x|-|b(x)|,
 \]
 and the claim follows.

Since
\[
\frac{d}{ds}\phi_s(\phi_{-t}(x))=b(\phi_s(\phi_{-t}(x))),
\]
one has
\[
\frac{d}{ds}|\phi_{s}(\phi_{-t}(x))|^2\le 2L|\phi_{s}(\phi_{-t}(x))|^2+2|b(0)| |\phi_{s}(\phi_{-t}(x))|.
\]
It follows that
\[
|x|\le |\phi_{-t}(x)|e^{Lt}+|b(0)|(e^{Lt}-1)
\Rightarrow
|\phi_{-t}(x)| \geq |x|e^{-Lt}-C(1-e^{-Lt}).
\]
Consider the backward flow. Differentiating explicitly,
\[
\frac{d}{dt}|\phi_{-t}(x)|^2 
= 2\phi_{-t}(x)\cdot \frac{d}{dt}\phi_{-t}(x)
= -2\phi_{-t}(x)\cdot b(\phi_{-t}(x)) \le 2L|\phi_{-t}(x)|^2 + 2|b(0)|\,|\phi_{-t}(x)|.
\]
Applying Gr\"onwall's inequality yields
\[
|\phi_{-t}(x)| \le |x|e^{Lt}+C(e^{Lt}-1).
\]

We note that
\begin{multline*}
\frac{d}{ds}|\phi_{-s}(x)-\phi_{-s}(\phi_t(y))|^2=2(\phi_{-s}(x)-\phi_{-s}(\phi_t(y)))\cdot[-b( \phi_{-s}(x))+b(\phi_{-s}(\phi_t(y)))]\\
\le 2L|\phi_{-s}(x)-\phi_{-s}(\phi_t(y))|^2,
\end{multline*}
Applying Gr\"onwall's inequality for $s\in [0, t]$ gives the left hand side of the last inequality.
Similarly
\begin{gather*}
\frac{d}{ds}|\phi_{s}(\phi_{-t}(x))-\phi_{s}(y)|^2\le 2L |\phi_{s}(\phi_{-t}(x))-\phi_{s}(y)|^2,
\end{gather*}
Applying Gr\"onwall's inequality for $s\in [0, t]$ then gives the right hand side of the last inequality.
\end{proof}

The following gives the properties of the heat semigroup.
\begin{lemma}\label{lem:diffu-step}
Let the initial density satisfy $A_1\exp(-\alpha |x|^p)\le \rho_0(x) \leq Ce^{-c|x|^2}$ where $C,c>0$ and $p\ge 2$. Then 
\begin{gather}
e^{t\cL_{\xi(2)}}\rho(x,t) \leq \frac{C}{(1+4c\beta^{-1} t)^{d/2}}e^{-\frac{c|x|^2}{1+4c\beta^{-1} t}},
\end{gather}
and
\begin{gather}
e^{t\cL_{\xi(2)}}\rho(x,t)\ge A_1\exp\left(-\frac{\beta^{-1}(d+p-2)^2}{4}t \right) \exp(-\alpha |x|^p).
\end{gather}
\end{lemma}

\begin{proof}
The Gaussian upper bound is a direction computation. 

Now, we consider the lower bound. It is clear that
\begin{gather}
e^{t\cL_{\xi(2)}}\rho(x,t) =\E \rho(x+\sqrt{2\beta^{-1}}W_{t}).
\end{gather}
Let $Y_t=x+\sqrt{2\beta^{-1}}W_{t}$. By the assumption, one then has
\[
e^{t\cL_{\xi(2)}}\rho(x,t) \ge A_1\E \exp(-\alpha|Y_t|^p).
\]
Setting $u(t)=\E \exp(-\alpha|Y_{t}|^p)$ and applying It\^o's formula, one has
\[
d|Y_t|^p=\sqrt{2\beta^{-1}} p |Y_t|^{p-2}Y_t\cdot dW+\beta^{-1}p(d+p-2)|Y_t|^{p-2}dt.
\]
Hence,
\begin{gather*}
du(t)=\beta^{-1}\E e^{-\alpha|Y_{t}|^p} ( \alpha^2p^2|Y_t|^{2p-2}-\alpha p(d+p-2)|Y_t|^{p-2})
\ge -\frac{\beta^{-1}(d+p-2)}{4}u(t).
\end{gather*}
Hence, for the lower bound in the diffusion step, the constant $\alpha_1$ can be taken to be unchanged,
while the constant $A_1$ is decreased by a factor $e^{-C\tau}$.

\end{proof}

Now, we prove Proposition \ref{pro:densityestimate}.
\begin{proof}[Proof of Proposition \ref{pro:densityestimate}]

We do by induction. Assume that 
\[
A_3^{(n)}\exp(-\alpha_1^{(n)} |x|^p)\le \rho_n^{\xi} \le A_4^{(n)}\exp(-\alpha_2^{(n)} |x|^2),
\]
we establish this for $\rho_{n+1}^{\xi}$.

We first consider the upper bound. If the advection step is first, then
\[
e^{\tau \cL_{\xi(1)}}\bar{\rho}_n\le A_4^{(n)}\exp(-\alpha_2^{(n)}|\phi_{-\tau}(x)|^2) J(-\tau, x).
\]
By Lemma \ref{lem:flowmap}, 
\[
 |\phi_{-\tau}(x)| \ge |x|e^{-C\tau}-C\tau,
\]
and $J(-\tau,x)\le e^{C\tau}$ since the drift $b$ is globally Lipschitz. 
Then, for $\tau\le 1$,
\[
|\phi_{-\tau}(x)|^2 \ge |x|^2e^{-2C\tau}(1-C_1\tau)-C_2\tau.
\]
This inequality holds for $|x|\le C\tau e^{C\tau}$ as well since the left hand side is obviously nonnegative. Then, after this step, one may take
\[
\tilde{A}_4^{(n+1)}=A_4^{(n)}e^{C\tau}e^{\alpha_2^{(0)}C_2\tau},
\quad \tilde{\alpha}_2^{(n+1)}=\alpha_2^{(n)}e^{-2C\tau}(1-C_1\tau),
\]
since $\max\{\tilde{\alpha}_2^{(n)}, \alpha_2^{(n)} \}\le \alpha_2^{(0)}$.
For the diffusion (heat) step, Lemma \ref{lem:diffu-step} tells us that 
we may take
\[
A_4^{(n+1)}=\tilde{A}_4^{(n+1)},\quad \alpha_2^{(n+1)}
=\frac{\tilde{\alpha}_2^{(n+1)}}{1+4\alpha_2^{(0)} \beta^{-1}\tau},
\]
since $\frac{1}{(1+4\tilde{\alpha}_2^{(n+1)}\beta^{-1}\tau)^{d/2}}\le 1$,
and again $\max\{\tilde{\alpha}_2^{(n)}, \alpha_2^{(n)} \}\le \alpha_2^{(0)}$.
If the diffusion step is first, the estimate is similar.
Then, one can take
\[
A_4^{(n+1)}=A_4^{(n)}e^{C\tau}e^{\alpha_2^{(0)}C_2\tau},
\quad \alpha_2^{(n+1)}
=\frac{1}{1+4\alpha_2^{(0)} \beta^{-1}\tau}e^{-2C\tau}(1-C_1\tau)\alpha_2^{(n)}.
\]
Hence, after $n$ steps up to finite horizon $T=n\tau$, the Gaussian upper bound is preserved, with $A_4$ enlarged at most by $(1+C\tau)^n \le C(T)$ and $\alpha_2$ decreased at most by $e^{-C\tau n}\ge e^{-CT}$.

\vspace{2mm}

Next we consider the lower bound. For the advection step, one has
\[
e^{\tau \cL_{\xi(1)}}\bar{\rho}_n\ge A_3^{(n)}\exp(-\alpha_1^{(n)}|\phi_{-\tau}(x)|^p) J(-\tau, x).
\]
Lemma \ref{lem:flowmap} ensures that
\[
|\phi_{-\tau}(x)|\le |x|e^{C\tau}+C\tau
\]
and similarly
\[
J(-\tau,x)\ge e^{-C\tau}.
\]
Then,
\[
|\phi_{-\tau}(x)|^p=|x|^p e^{Cp\tau}+\sum_{k=1}^p {p\choose k}|x|^{p-k}
e^{C(p-k)\tau}C^k\tau^k \le |x|^p e^{Cp\tau}(1+C\tau)+C\tau,
\]
where $C$ here depends on $p$. Hence, one may take
\[
\tilde{A}_3^{(n+1)}=A_3^{(n)}\exp(-\alpha_1^{(n)}C\tau)e^{-C\tau},
\quad \tilde{\alpha}_1^{(n+1)}=\alpha_1^{(n)}(1+C\tau).
\]
In the diffusion step, Lemma \ref{lem:diffu-step} implies that we can take 
\[
A_3^{(n+1)}=\tilde{A}_3^{(n+1)}\exp(-\frac{\beta^{-1}}{4}(d+p-2)^2\tau),
\quad \alpha_1^{(n+1)}=\tilde{\alpha}_1^{(n+1)}.
\]
If the diffusion step is first, the estimate would be similar. Hence, no matter which case, one has
\[
A_3^{(n+1)}=A_3^{(n)}\exp(-\alpha_1^{(n)}C\tau)e^{-C\tau}\exp(-\frac{\beta^{-1}}{4}(d+p-2)^2\tau),
\quad \alpha_1^{(n+1)}=\alpha_1^{(n)}(1+C\tau).
\]
Using this recursion relation, $\alpha_1^{(n)}\le \alpha_1^{(0)}\exp(CT)$. Then, this further implies that $A_3^{(n)}
\ge A_3^{(0)}/C(T)$ for all $n\tau\le T$.

This then shows that the upper and lower bounds can be propagated and the claims hold.
\end{proof}

\subsection{Derivatives of the logarithmic density}\label{subsec:derivativebound}

This subsection derives uniform bounds on the first and second-order derivatives of the numerical logarithmic density. These regularity estimates provide the necessary technical tools for establishing the fourth-order convergence rate in relative entropy.

\subsubsection{First order derivative estimates}

Similarly as above, we consider
\begin{gather}
v_n^{\xi}:=\frac{|\nabla f_n^{\xi}|^2}{(1-f_n^{\xi}+\log M)^2},
\quad f_n^{\xi}=\log\rho_n^{\xi}.
\end{gather}
Here, $M$ is an upper bound of $\rho_n^{\xi}$ on $[0, T]$. 
By the assumption given, it is clear that
\begin{gather}
A=\sup_{x} v_0(x)<\infty.
\end{gather}
Similarly, we can set $M=1$ without loss of generality.

Next, we aim to argue that this quantity stays bounded stated as follows.
\begin{proposition}\label{pro:pointwiseest}
Suppose Assumptions \ref{ass:drift} and \ref{ass:initial-density} hold. Then, it holds that
\begin{gather}
\sup_x\sup_{t_n\le T} v_n^{\xi}(x)\le C(T),
\end{gather}
where $C(T)$ is independent of $\xi$.
\end{proposition}

To prove this claim, we first check the effects of both the advection step and the diffusion step. Then, we get the final estimate.

For the advection step, we have the following one step stability.
\begin{lemma}\label{lem:advection}
Let $\tilde{\rho}=e^{\tau \cL_{\xi(1)}}\rho$ and $\tilde{v}$  be the $v$ corresponding to $\tilde{\rho}$. Then, one has
\begin{gather}
\sup_x \tilde{v}\le (1+C\tau)\sup_x v+C\tau.
\end{gather}
\end{lemma}

\begin{proof}
By the assumption, since $M=1$ is the uniform bound, which implies that
\[
e^{\tau \cL_{\xi(1)}}\rho=\rho(\phi_{-\tau}(x))J(-\tau, x)\le 1.
\] 
Hence, $1-f(\phi_{-\tau}(x))-\log J(-\tau, x)\ge 1$.

By the fact that
\[
|\nabla \phi_{-\tau}(x)|\le (1+C\tau), \quad |\nabla\log J(-\tau, x)|\le C\tau,
\]
one has
\begin{gather*}
\begin{split}
\tilde{v} &=\frac{|\nabla\phi_{-\tau}(x)\cdot\nabla f(\phi_{-\tau}(x))+\nabla \log J(-\tau, x)|^2}{(1-f(\phi_{-\tau}(x))-\log J(-\tau, x))^2}\\
&\le \frac{(1+C\tau)^2|\nabla f(\phi_{\tau}(x))|^2(1+C\tau)+C\tau}{(1-f(\phi_{-\tau}(x)))^2 [1-\log J(-\tau,x)/(1-f(\phi_{-\tau}))]^2}.
\end{split}
\end{gather*}
Here, we used $(a+C\tau)^2 \le (1+C\tau)a^2 + C\tau$ for $\tau\le 1$.  
Moreover, since $1-f(\phi_{-\tau}(x))\ge 1$, we have $\bigl|\tfrac{\log J(-\tau,x)}{1-f(\phi_{-\tau}(x))}\bigr|\le C\tau$, so that for $\tau\le 1$,
\[
\tilde{v}(x)\le 
(1+C\tau)\frac{|\nabla f(\phi_{\tau}(x))|^2}{(1-f(\phi_{-\tau}(x)))^2}
+C\tau.
\]
Hence, the claim follows.
\end{proof}

Next, we consider the diffusion step.
\begin{lemma}\label{lem:diffusion-step}
Let $\tilde{\rho}=e^{\tau \cL_{\xi(2)}}\rho$ and $\tilde{v}$ be the $v$ corresponding to $\tilde{\rho}$. Then,
\[
\sup_x \tilde{v}\le \max\{\sup_x  v  , A_1\},
\]
where $A_1$ is independent of $\xi, n, \tau$ and the initial value.
\end{lemma}

The proof of Lemma \ref{lem:diffusion-step} follows analogously to Theorem \ref{thm:gradient_estimate}, where 
we derived estimates for $f = \log\rho$ in the continuous setting (with $\rho$ satisfying $ 
f_t= \beta^{-1} f_{i i}+ \beta^{-1} f_i f_i$, i.e., $b^i = 0$, $c = 0$). We omit the details.

\begin{proof}[Proof of Proposition \ref{pro:pointwiseest}.]
Combining Lemma \ref{lem:advection} and Lemma \ref{lem:diffusion-step}, for each time step $\tau$, one can directly derive the desired bound for $v$.
\end{proof}

\subsubsection{Second order derivative estimates}

Now, we have proved that $v_n^{\xi}$ is bounded for $t\le T$ uniformly in $\xi$. Then, we can find $A>0$ such that the following $w_n^{\xi}$, 
in analogy with the continuous case, is well-defined
\[
w_n^{\xi} := \frac{1}{q(x)} \frac{|\nabla^2 f_n^{\xi}|^2}{A(1-f_n^{\xi}+\log M)^2 - |\nabla f_n^{\xi}|^2}, \quad f_n^{\xi}=\log \rho_n^{\xi}(x)
\]
and $q(x) = 1+|x|^{r}$ for some $r\ge 0$ to be chosen so that $w$ is bounded at $n=0$ and extra polynomial terms can be controlled.

\begin{proposition}[second order derivative estimate]\label{pro:second_order_main}
Under the assumptions of Proposition \ref{pro:pointwiseest}, there exists some polynomial $q(x)=1+|x|^r$ such that
	\[
	\sup_{n: n\tau \leq T} \sup_x w_n^{\xi}(x) \leq C(T),
	\]
    where $C(T)$ is independent of $\xi$.
\end{proposition}

We first consider the effect of the advection step.
\begin{lemma}\label{lem:second_order_advection}
Let $\tilde{\rho}=e^{\tau \cL_{\xi(1)}}\rho$ where $\rho$ is any possible $\rho_n^{\xi}$ and $\tilde{w}$  be the $w$ corresponding to $\tilde{\rho}$. Then there exists $C>0$ independent of $\tau$ such for $\tau\le 1$ that
\begin{gather}
\sup_x \tilde{w} \leq (1 + C\tau)\sup_x w + C\tau.
\end{gather}
\end{lemma}

\begin{proof}
First, one has
\[
\tilde{f}=\tilde{f}(\phi_{-\tau}(x))+\log J(-\tau, x).
\]
	
It is clear that
\[
\nabla^2\tilde{f}=\nabla\phi_{-\tau}(x)^{\otimes 2}:\nabla^2 f(\phi_{-\tau}(x))
+\nabla^2\phi_{-\tau}(x)\cdot\nabla f(\phi_{-\tau}(x))+\nabla^2J(-\tau, x).
\]
Here 
\[
(\nabla\phi_{-\tau}(x)^{\otimes 2}:\nabla^2 f(\phi_{-\tau}(x)))_{ij}
:=\partial_i\phi_m \partial_j\phi_n \partial_{mn}f.
\]
By Lemma \ref{lem:flowmap}, the flow $\phi_{-\tau}$ satisfies the estimates
\[
|\nabla\phi_{-\tau}(x)-I|\le C\tau, 
\qquad \|\nabla^2\phi_{-\tau}\|+|\nabla^2J(-\tau, x)|\le C\tau.
\]
Again, we can rescale all $\rho_n^{\xi}$ by $M$ so that we can set $M=1$
and similarly, $1-f(y)\ge 1$ and $1-f(\phi_{-\tau}(x))-\log J(-\tau, x)\ge 1$.

By Young's inequality, one then has for $\tau\le 1$ that
\begin{equation}\label{eq:hessian_bound}
    \|\nabla^2\tilde{f}\|^2\le (1+C\tau)\|\nabla^2 f(\phi_{-\tau}(x))\|^2
    +C\tau \|\nabla f(\phi_{-\tau}(x))\|^2+C\tau.
\end{equation}

For the denominator $D(\tilde{f}) := A(1-\tilde{f})^2 - |\nabla\tilde{f}|^2$, one has
\begin{multline*}
D(\tilde{f}) =A(1-\tilde{f}(\phi_{-\tau}(x))-\log J(-\tau, x))^2
-|\nabla\phi_{-\tau}(x)\cdot \nabla f(\phi_{-\tau}(x))+\nabla\log J(-\tau, x)|^2\\
\ge 
A(1-\tilde{f}(\phi_{-\tau}(x)))^2-|\nabla f(\phi_{-\tau}(x))|^2
-C(1-\tilde{f})\tau-C\tau|\nabla f|^2-C\tau|\nabla f|.
\end{multline*}
As we have constructed, by the boundedness of $v$,
\[
A(1-\tilde{f}(\phi_{-\tau}(x)))^2-|\nabla f(\phi_{-\tau}(x))|^2\ge
\frac{A}{2}(1-\tilde{f}(\phi_{-\tau}(x)))^2\ge \max\{\frac{A}{2},
\frac{1}{2}|\nabla f(\phi_{-\tau}(x))|^2 \}.
\]
It is straightforward to see that
\[
D(\tilde{f})\ge \left( A(1-\tilde{f}(\phi_{-\tau}(x)))^2-|\nabla f(\phi_{-\tau}(x))|^2 \right) (1-C\tau).
\]
Hence, one has
\[
\tilde{w}(x)\le \frac{q(\phi_{-\tau}(x))}{q(x)} w(\phi_{-\tau}(x))(1+C\tau)
+\frac{1}{q(x)}\frac{C\tau |\nabla f(\phi_{-\tau}(x))|^2+C\tau}{\max\{\frac{A}{2}, \frac{1}{2}|\nabla f(\phi_{-\tau}(x))|^2 \}}
\]
Next, by Lemma \ref{lem:flowmap}, one finds that
\[
\frac{q(\phi_{-\tau}(x))}{q(x)} \le 1+C\tau,
\]
for $\tau\le 1$ where $C$ depends on $r$. The second term is clearly bounded by $C\tau$.
\end{proof}

Next, we consider the diffusion step. 
\begin{lemma}\label{lem:second_order_iterative}
Let $\tilde{\rho}=e^{\tau \cL_{\xi(2)}}\rho$ and $\tilde{w}$ be the $w$ corresponding to $\tilde{\rho}$. Then, there is some polynomial $q(x)=1+|x|^r$ so that for a constant $A_2$ independent of $\xi, n,\tau$ and the initial value,  it holds for all possible $\rho=\rho_n^{\xi}$ that
\[
\sup_{x} \tilde{w} \leq \max\left\{\sup_x w, A_2\right\}.
\]
\end{lemma}

The proof of Lemma \ref{lem:second_order_iterative} follows analogously to Theorem \ref{thm:gradient_estimate}.
\begin{proof}[Proof of Proposition \ref{pro:second_order_main}]
This is a direct consequence of Lemma \ref{lem:second_order_advection} and Lemma \ref{lem:second_order_iterative}.
\end{proof}

\subsection{The polynomial bounds of the derivatives for the logarithmic density}\label{subsec:boundsforrhobar}

Based on the derivative bounds for $\rho_n^{\xi}$ established above, 
we can extend these estimates to the averaged density $\bar\rho(t)$.

\begin{theorem}\label{thm:deriva_density}
Consider the averaged interpolated density $\bar{\rho}(t)$ for the random splitting Langevin Monte Carlo defined in \eqref{eq:barrhot}.  
Suppose both Assumptions \ref{ass:drift} for the coefficients and \ref{ass:initial-density} for the initial density hold. Then, for any $T>0$, there exist some $p>0$ and $C(T)>0$ satisfying
\begin{equation}
\sup_{t\le T}\left( |\log \bar{\rho}(t,x)|+|\nabla \log \bar\rho(t,x)|+
|\nabla^2 \log \bar\rho(t,x)|\right) \;\le\; C(T)\,(1+|x|)^p.
\end{equation}
\end{theorem}

\begin{proof}
First, we remark that the results in Proposition \ref{pro:densityestimate}, Proposition \ref{pro:pointwiseest} and 
Proposition \ref{pro:second_order_main} can be extended to $\rho^{\xi}(t)$ defined in \eqref{eq:timeinterpolation0}. In fact, this can be understood that the last step is performed with step size $t-t_n$ and the proof has no difference.

First, by Proposition \ref{pro:densityestimate}, $\bar{\rho}(t)$
has the same upper and lower bound and thus the claim for $\log\bar{\rho}$ holds.

Moreover, by the boundedness of $v^{\xi}(t)$ in Proposition \ref{pro:pointwiseest},
\[
|\nabla\log \rho^{\xi}(t)|\le C(1-\log\rho^{\xi}(t)+\log M)
\le C(1+|x|^p).
\]
Since the mapping
$(\rho, m) \mapsto m^q/\rho^{q-1}$ for $q\ge 1$ is convex, one thus has by Jensen's inequality that
\[
\bar{\rho}(t) |\nabla \log \bar{\rho}(t)|^q
\le \E_{\xi}[\rho^{\xi}(t)|\nabla\log \rho^{\xi}(t)|^q]
\le \E_{\xi}[\rho^{\xi}(t)(C(1+|x|^p))^q]=\bar{\rho}(t)(C(1+|x|^p))^q.
\]
This gives the result for $|\nabla \log \bar{\rho}(t)|$.

Lastly, by the same argument,
\[
\frac{|\nabla^2\bar{\rho}(t)|^q}{(\bar{\rho}(t))^{q-1}}
\le \E_{\xi}\left[\frac{|\nabla^2 \rho^{\xi}(t)|^q}{(\rho^{\xi}(t))^{q-1}}\right]
=\E_{\xi}[\rho^{\xi}(t)|\nabla^2\log\rho^{\xi}(t)
+\nabla\log\rho^{\xi}\otimes \nabla\log\rho^{\xi}|^q].
\]
Hence, one has
\begin{multline*}
\bar{\rho}(t)|\nabla^2\log\bar{\rho}(t)|^q
\le \E_{\xi}[\rho^{\xi}(t)|\nabla^2\log\rho^{\xi}(t)
+\nabla\log\rho^{\xi}\otimes \nabla\log\rho^{\xi}|^q]
-\bar{\rho}(t)|\nabla\log\bar{\rho}(t)|^{2q}.
\end{multline*}
Then, by the estimate of $w^{\xi}$ in Proposition \ref{pro:second_order_main}, one finds easily that 
$|\nabla^2\log\rho^{\xi}(t)|$ is bounded by a polynomial uniformly in $\xi$ as well. Hence, by the same argument, the bounds 
on $\nabla^2\log\bar{\rho}$ follows.
\end{proof}

\section{The sampling error estimates}\label{sec:sampling-error}

In this section we address the long-time behavior of the splitting scheme. 
While the previous section established local error estimates on finite time intervals (Theorem \ref{thm:uniform_error}), 
our goal here is to derive global error bounds that remain uniform in time. 
The key observation is that the splitting scheme itself admits a unique invariant distribution and is ergodic with exponential convergence. 
By combining this ergodicity property with the finite-time local error analysis, 
we obtain a sharp estimate of the distance between the invariant distribution of the numerical scheme and the invariant distribution of the continuous dynamics.

Let $\cS_n^{\tau}$ be the evolutional operator for the laws of the random splitting LMC, i.e.,
\[
\cS_n^{\tau}\rho_0:=\mathrm{Law}(X_n),~ \text{if}~ \rho_0=\mathrm{Law}(X_0).
\]
Clearly, $\cS_n^{\tau}$ is a linear operator. Moreover, by the Markov property, 
\begin{gather}
\cS_n^{\tau}=(\cS^{\tau})^n,
\end{gather}
so that $\{\cS_n^{\tau}\}$ forms a Markov semigroup.

\begin{theorem}[Ergodicity of the splitting scheme]
\label{thm:ergodicity}
Under Assumptions \ref{ass:drift} and \ref{ass:initial-density}, 
the Markov semigroup $S^\tau$ associated with the random splitting LMC admits a unique invariant distribution $\bar{\rho}_*^\tau$. 
Moreover, there exist constants $C,\lambda>0$, independent of $\tau$, such that for any probability measure $\mu$,
\begin{equation}
    W_1\!\big((S^\tau)^n \mu,\, \bar{\rho}_*^\tau \big) \;\leq\; C e^{-\lambda n\tau}\, W_1(\mu,\, \bar{\rho}_*^\tau),
\end{equation}
for all integers $n\geq 0$.
\end{theorem}

\begin{proof}

By Theorem \ref{thm:ergodicity0}, one finds that there exists some $n_0>0$ such that $(\cS^{\tau})^{n_0}$ is a contraction under the $W^1$ metric. By the contraction principle, $(\cS^{\tau})^{n_0}$ has a unique invariant measure $\bar{\rho}_*^{\tau}$. Then, it can be verified directly that
\[
\tilde{\pi}^{\tau}:=\frac{1}{n_0}\sum_{i=0}^{n_0-1}
(\cS^{\tau})^{i}\bar{\rho}_*^{\tau}
\]
is an invariant measure of $\cS^{\tau}$. 
Moreover, any invariant measure of $\cS^{\tau}$ is also an invariant measure of $(\cS^{\tau})^{n_0}$ and thus must be $\bar{\rho}_*^{\tau}$ by the uniqueness of the invariant measure of $(\cS^{\tau})^{n_0}$, concluding the first claim.

The second claim follows from Theorem~\ref{thm:ergodicity0} directly by setting $\nu=\bar{\rho}_*^{\tau}$.
\end{proof}

We now combine the finite-time local error estimate (Theorem \ref{thm:uniform_error}) with the ergodicity result (Theorem \ref{thm:ergodicity}). 
This yields a global error estimate for the random splitting LMC uniformly in time, thus yielding the error for the invariant measure. This gives the support for sampling using the random splitting LMC.

\begin{theorem}
\label{thm:global_error}
Let $\rho(t) = \cS(t)\rho_0$ denote the solution of the Fokker--Planck equation with initial distribution $\rho_0$, and let $\bar{\rho}_n = (\cS^\tau)^n \rho_0$ be the law generated by the splitting scheme.  Then there exists a constant $C>0$, independent of $n$ and $\tau$, such that the following hold:
\begin{equation}\label{eq:uniform-in-time}
    W_1\!\big(\cS(\tau)^n \rho_0,\,(\cS^\tau)^n \rho_0\big) \;\le\; C \tau^2,
    \qquad \text{for all } n\ge 0.
\end{equation}
Moreover, the invariant measure $\bar\rho_*^\tau$ satisfies
\begin{equation}\label{eq:invariant-measure-error}
    W_1(\bar{\rho}_*^\tau, \rho_\ast) \;\le\; C \tau^2,
\end{equation}
where $\rho_\ast$ is the unique probability invariant measure of the continuous dynamics.
\end{theorem}

\begin{proof}
We first establish the uniform-in-time estimate \eqref{eq:uniform-in-time}.  
From the local error analysis in relative entropy and a suitable transport--entropy inequality, one obtains that for any fixed finite horizon $T>0$,
\[
W_1\!\big(\cS(\tau)^n \rho,\,(\cS^\tau)^n \rho \big) \;\le\; C(T)\,\tau^2,
\qquad \text{for all } n\tau \le T,
\]
for any $\rho$ satisfying Assumption \ref{ass:initial-density} where $C(T)$ depends on $T$ and $\rho$ through the coefficients in Assumption \ref{ass:initial-density} only, and is independent of $\tau$.  

To get the uniform-in-time estimate, we combine the finite-time estimate with the exponential ergodicity of the numerical dynamics, together with the uniform estimates of the density for the time-continuous dynamics.  

By the exponential ergodicity of the random splitting LMC (Theorem~\ref{thm:ergodicity0}), there exists $T_0>0$ such that whenever $n\ge T_0/\tau$ and any probability measures $\mu,\nu$
\[
W_1\!\big((\cS^\tau)^{n}\mu,\,(\cS^\tau)^{n}\nu\big) \;\le\; \gamma' W_1(\mu,\nu), 
\]
where $\gamma'$ is independent of $\tau$.

Set $n_0 = \lceil T_0/\tau\rceil$ and take $n = k n_0$, $m=(k-1)n_0$. Using the semigroup property, one has
\begin{multline*}
W_1\!\big(\cS(\tau)^n \rho_0,\,(\cS^\tau)^n \rho_0\big) \\
\le W_1\!\big(\cS(\tau)^{n-m}\rho(t_m),\,(\cS^\tau)^{n-m}\rho(t_m)\big)
+ W_1\!\big((\cS^\tau)^{n-m}\rho(t_m),\,(\cS^\tau)^{n-m}\bar\rho_m\big).
\end{multline*}

Since $(n-m)\tau=n_0\tau\le T_0+\tau_0$ with $\tau_0$ being the maximum step size,  the first term on the right-hand side is bounded by $C(T_0+\tau_0)\tau^2$ and the constant $C(T_0+\tau_0)$ is independent of $k$ since $\rho(t_m)$ has uniform estimates by Proposition \ref{pro:density-SDE} and Theorem \ref{thm:gradient_estimate}.  
The second term is bounded by
\[
W_1\!\big((\cS^\tau)^{n-m}\rho(t_m),\,(\cS^\tau)^{n-m}\bar\rho_m\big) 
\;\le\; \gamma' W_1(\rho(t_m),\,\bar\rho_m).
\]
Hence,
\[
W_1\!\big(\cS(\tau)^{kn_0} \rho_0,\,(\cS^\tau)^{kn_0} \rho_0\big) \;\le\; C \tau^2+\gamma' W_1\!\big(\cS(\tau)^{(k-1)n_0} \rho_0,\,(\cS^\tau)^{(k-1)n_0} \rho_0\big).
\]
which establishes \eqref{eq:uniform-in-time} for $n=kn_0$.
For general $n$, one just applies the local in time error estimate again starting from the nearest integer of the form $kn_0$.

Finally, we pass to the limit $n\to\infty$. By ergodicity of the continuous and random splitting LMC dynamics,
one has \eqref{eq:invariant-measure-error}.
\end{proof}



\section{Numerical Experiments}\label{sec:experiments}

This section presents numerical evidence for the fourth-order convergence of the randomized splitting Langevin method (RSLMC) in relative entropy. We study three one-dimensional benchmarks and one two-dimensional multimodal example, chosen to cover light-tailed non-Gaussian targets, non-convex multi-well energy landscapes, a rigorously constructed model that satisfies all assumptions of the theory, and a higher-dimensional mixture model. Unless otherwise stated, we use $M=10^7$ Monte Carlo trajectories, a final integration time $T=50$, and inverse temperature $\beta=1$. For each step size $\tau$, the endpoint distribution generated by RSLMC is compared against a reference distribution obtained by either exact sampling (inverse transform), rejection sampling with a Gaussian proposal, or direct mixture sampling.

Probability densities are approximated by Gaussian kernel density estimation (KDE). Given samples $\{X^{(j)}\}_{j=1}^M$ in dimension $d$, the estimator is
\[
\hat\rho(x) = \frac{1}{M\,h^d}\sum_{j=1}^M (2\pi)^{-d/2}\exp\!\left(-\frac{\|x-X^{(j)}\|^2}{2h^2}\right),
\]
where the bandwidth $h$ is chosen automatically by Silverman's rule of thumb
\[
h = \sigma\left(\frac{4}{d+2}\cdot\frac{1}{M}\right)^{1/(d+4)},
\]
with $\sigma$ is the sample standard deviation. In one-dimensional tests, the selected bandwidth is further scaled by a constant factor to suppress spurious oscillations; in the two-dimensional case, the unmodified rule is applied. All estimated densities are normalized over the computational grid used for entropy integration, and a small floor $\epsilon=10^{-12}$ is enforced inside logarithms to stabilize the KL computation.

In RSLMC, the deterministic ODE substep follows the drift flow $\dot x=-\nabla U(x)$. When a closed-form flow is not available, we discretize this ODE by the second-order Heun (explicit trapezoidal) method:
\[
Y = X + \tau f(X),\qquad X_{\text{new}} = X + \tfrac{\tau}{2}\big(f(X)+f(Y)\big),
\]
with $f(x)=-\nabla U(x)$. For the double-well potential, the drift admits an analytic Strang splitting: write $f(x)=-4x^3+4x=f_1(x)+f_2(x)$, whose subflows are
\[
\phi_h^{(1)}(x)=\frac{x}{\sqrt{1+8h\,x^2}},\qquad \phi_h^{(2)}(x)=x\,e^{4h}.
\]
A step of size $h$ is then the exact composition $\Phi_h=\phi_{h/2}^{(1)}\circ \phi_h^{(2)}\circ \phi_{h/2}^{(1)}$.  

\medskip

\subsection{One-dimensional test cases}

\paragraph{(a) Quadratic--logcosh model (compliant example).}
The first experiment considers the potential
\[
U(x) = \tfrac{\lambda}{2}x^2 + \varepsilon \log(2\cosh x),
\]
with parameters $\lambda=1$ and $\varepsilon=0.8$. The quadratic term enforces strong convexity away from the origin, while the $\log\cosh$ perturbation has bounded derivatives, so the model satisfies all assumptions of the theoretical framework. Reference samples are drawn by rejection sampling using a Gaussian proposal matched to the quadratic part. The numerical simulations use Heun drift steps with $\tau \in \{1,2^{-1},\,2^{-2},\,2^{-3},\,2^{-4}\}$. Kernel densities are estimated on a 512-point grid covering the central $[10^{-4},\,1-10^{-4}]$ quantiles; the Silverman's bandwidth is scaled by a factor of two.

\begin{figure}[!ht]
    \centering
    \includegraphics[width=0.6\linewidth]{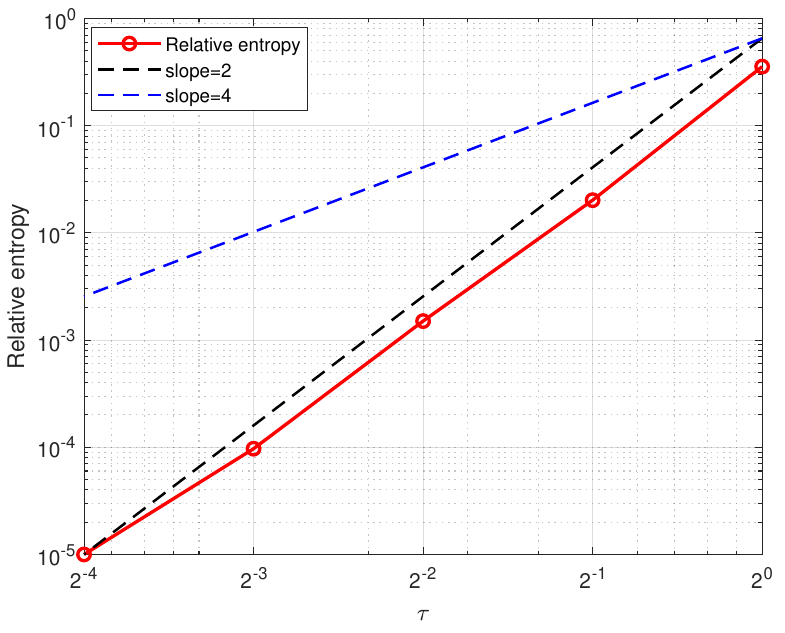}
    \caption{Relative entropy versus step size for the quadratic--logcosh model (log--log scale). 
    The empirical curve (solid line with markers) is compared with reference slopes $\tau^2$ and $\tau^4$ (dashed).}
    \label{fig:compliant}
\end{figure}

Figure~\ref{fig:compliant} demonstrates fourth-order convergence for the quadratic--logcosh model, which fully satisfies the theoretical assumptions. 
The observed KL divergence decays from $3.5\times 10^{-1}$ at $\tau=1$ to $10^{-5}$ at $\tau=2^{-4}$, with the empirical curve matching the $\tau^4$ reference and diverging clearly from a second-order trend. 
This provides direct numerical evidence that the sharp theoretical error bound is realized in practice. 
The tested step sizes $\tau \in \{1,2^{-1},\,2^{-2},\,2^{-3},\,2^{-4} \}$ are selected to cover both coarse scales, where higher-order accuracy is most visible, and fine scales, where the asymptotic error constant can be quantified.

\paragraph{(b) Double-well potential.}
The second example is the quartic double-well potential
\[
U(x) = (x^2-1)^2,
\]
whose Gibbs distribution is bimodal with metastable wells around $\pm1$. This non-convex landscape tests stability and accuracy under non-globally Lipschitz drifts. Reference samples are produced by rejection sampling with a Gaussian proposal. The RSLMC drift step uses the analytic Strang splitting described above; time steps are $\tau=2^{-4},\,2^{-5},\,2^{-6},\,2^{-7},\,2^{-8}$. For KDE, we fix a uniform grid of 1024 points on $[-4,4]$ and use Silverman's bandwidth.

\begin{figure}[!ht]
    \centering
    \includegraphics[width=0.6\linewidth]{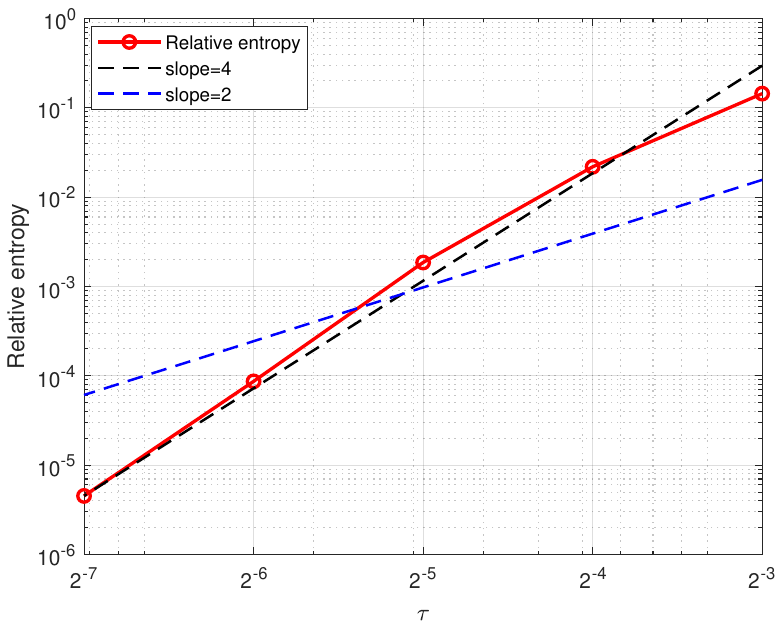}
    \caption{Relative entropy versus step size for the 1D double-well potential (log--log scale). 
    The empirical curve (solid line with markers) is compared with reference slopes $\tau^2$ and $\tau^4$ (dashed).}
    \label{fig:doublewell}
\end{figure}

Numerical results in Figure~\ref{fig:doublewell} confirm the predicted fourth-order decay of the KL divergence for the double-well potential. 
Across the tested step sizes, the empirical curve remains nearly parallel to the $\tau^4$ line and is clearly separated from a second-order slope. 
The KL divergence decreases from approximately $1.5\times 10^{-1}$ at $\tau=2^{-4}$ to $4\times 10^{-6}$ at $\tau=2^{-8}$, reaching very small error magnitudes even in this metastable, non-convex landscape. 
The choice of dyadic step sizes down to $2^{-8}$ balances the need for sufficiently fine resolution to resolve metastable transitions while keeping the computational effort feasible.

\paragraph{(c) Logistic distribution.}
The  experiment considers the logistic distribution,
\[
\pi(x)\propto \frac{e^{-x}}{(1+e^{-x})^2},
\]
which has lighter-than-Gaussian tails and does not satisfy strong convexity. This case is included to illustrate the robustness of RSLMC when applied to non-standard target measures outside the strict theoretical framework. The reference distribution is generated via inverse-CDF sampling. The numerical side employs RSLMC with the Heun ODE step and step sizes $\tau \in \{0.2, 0.4, 0.6, 0.8, 1.0\}$. Densities are approximated on a uniform grid of 512 points covering the central quantile interval $[10^{-4}, 1-10^{-4}]$. The bandwidth is selected by Silverman's rule and multiplied by a factor of three to ensure smoothness of the KDE.

\begin{figure}[!ht]
    \centering
    \includegraphics[width=0.6\linewidth]{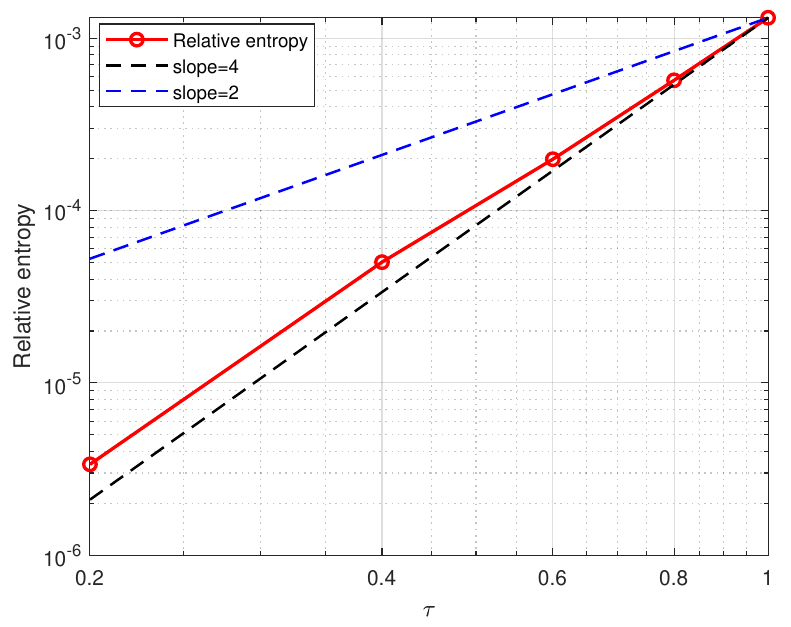}
    \caption{Relative entropy versus step size for the 1D logistic distribution (log--log scale). 
    The empirical curve (solid line with markers) is compared with reference slopes $\tau^2$ and $\tau^4$ (dashed).}
    \label{fig:logistic}
\end{figure}

Numerical analysis confirms fourth-order convergence in relative entropy for the logistic distribution. 
As seen in Figure~\ref{fig:logistic}, the empirical curve follows the $\tau^4$ reference slope closely and deviates significantly from the second-order line. 
The KL divergence decreases from about $2\times 10^{-3}$ at $\tau=1.0$ to $3\times 10^{-6}$ at $\tau=0.2$, demonstrating that RSLMC achieves accurate sampling even in a non-strongly-convex, light-tailed target distribution. 
The step sizes are chosen in the range $\tau\in[0.2,1]$ to highlight both the stability of the method at coarse resolution and its accuracy at practically relevant finer steps.

\subsection{Two-dimensional mixture of Gaussians}

To illustrate performance in higher dimensions, we consider a two-component Gaussian mixture distribution
\[
\pi(x) = \tfrac{1}{2}\mathcal{N}((-2,0), \Sigma_1) + \tfrac{1}{2}\mathcal{N}((2,0), \Sigma_2),
\]
with covariance matrices
\[
\Sigma_1 = \begin{pmatrix}0.6 & 0.2 \\ 0.2 & 0.5\end{pmatrix},\qquad
\Sigma_2 = \begin{pmatrix}0.5 & -0.1 \\ -0.1 & 0.7\end{pmatrix}.
\]
This multimodal target provides a stringent benchmark for sampling in higher-dimensional, non-convex settings. Reference samples are generated directly from the analytic mixture. The numerical side uses RSLMC with the Heun ODE step and step sizes $\tau\in\{0.1,0.2,0.4,0.6,0.8\}$. KDE is performed on a $300\times 300$ uniform grid spanning a bounding box defined by empirical quantiles of the reference samples. The bandwidth is determined by Silverman's rule.

\begin{figure}[!ht]
    \centering
    \includegraphics[width=0.6\linewidth]{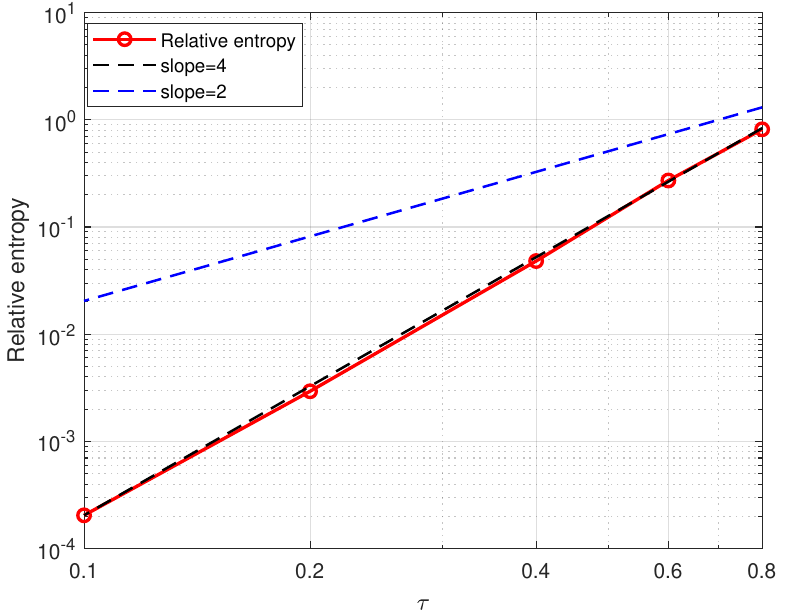}
    \caption{Relative entropy versus step size for the 2D Gaussian mixture (log--log scale). 
    The empirical curve (solid line with markers) is compared with reference slopes $\tau^2$ and $\tau^4$ (dashed).}
    \label{fig:mog2d}
\end{figure}

The two-dimensional Gaussian mixture experiment, reported in Figure~\ref{fig:mog2d}, further confirms fourth-order convergence of RSLMC in higher-dimensional and non-convex settings. 
The KL divergence decreases from $8.5\times 10^{-1}$ at $\tau=0.8$ to $2\times 10^{-4}$ at $\tau=0.1$, with the log--log slope closely aligned with the $\tau^4$ reference throughout the tested range. 
This demonstrates that RSLMC maintains its theoretical accuracy even in multi-modal distributions, where sampling is particularly challenging. 
The choice of step sizes $\tau\in\{0.1,0.2,0.4,0.6,0.8\}$ reflects a compromise between the cost of high-dimensional KDE estimation and the need to verify convergence over multiple scales.

\section*{Acknowledgement}

This work was financially supported by the National Key R\&D Program of China, Project Number 2021YFA1002800.
The work of L. Li was partially supported by NSFC 12371400, and  Shanghai Municipal Science and Technology Major Project 2021SHZDZX0102.

\appendix

\section{Estimates of the solution to Fokker-Planck}\label{app:esti-FP}

In this section, we establish bounds for the solution of the Fokker-Planck equation corresponding to the Langevin dynamics.

We consider
\[
q_t(x)=\rho_t(x)/e^{-\beta U(x)}.
\]
Then, 
\[
\partial_t q=-\nabla U\cdot \nabla q+\beta^{-1}\Delta q.
\]
It is well know that
\[
q_t(x)=\E q_0(X_t(x)),
\]
where $X_t(x)$ is the stochastic trajectory of the SDE
\[
dX=-\nabla U(X)\,dt+\sqrt{2\beta^{-1}}\,dW.
\]
Recall the claim that
\[
B_1\exp(-C|x|^p)  \le q_t(x)\le B_2\exp(\beta (1-\delta)  U(x)).
\]

\begin{proof}[ \textit{Proof of Proposition \ref{pro:density-SDE}} ]
By the assumption on $\rho_0$,
\[
B_1'\exp(-C'|x|^p)  \le q_0(x)\le B_2'\exp((1-\delta')\beta U).
\]

For the upper bound, we apply It\^o's formula to  $e^{\beta(1-\delta)U}$. In fact,
\[
dU(X)=(\beta^{-1}\Delta U-|\nabla U|^2)\,dt+\sqrt{2\beta^{-1}}\nabla U(X)dW.
\]
Then,
\[
\frac{d}{dt}\E e^{\beta(1-\delta)U(X)}=\E\left\{ e^{\beta(1-\delta)U(X)}[-\beta(1-\delta)\delta |\nabla U(X)|^2+(1-\delta)\Delta U(X)]\right\}.
\]
By the assumptions
\[
-\beta(1-\delta)\delta |\nabla U(X)|^2+(1-\delta)\Delta U(X)\le -\lambda +C1_{|X|\le L}.
\]
Hence,
\[
\frac{d}{dt}\E e^{\beta(1-\delta)U(X)}\le -\lambda \E e^{\beta(1-\delta)U(X)}+C.
\]
This implies that
\[
\sup_{t\ge 0}q_t(x)\le e^{\beta(1-\delta)U(x)}e^{-\lambda t}+C.
\]

For the lower bound, we show that there exists $L>0$ such that for all $x$ and $t$ that,
\[
\P(|X_t(x)|\le L(|x|+1)) \ge 1/2.
\]
In fact, 
by the Markov inequality
\begin{multline*}
\P(|X|\ge L(|x|+1))\le [e^{\beta(1-\delta)U(x)}e^{-\lambda t}+C]/\inf_{y\in \R^d\setminus B(0, L(1+|x|))}e^{\beta(1-\delta)U(y)}\\
\le \exp(\beta(1-\delta)[\lambda_2 |x|^2-\lambda_1|L+L|x||^2])e^{-\lambda t}
+C\exp(\beta(1-\delta)[-\lambda_1|L+L|x||^2]).
\end{multline*}
Then,
\[
\inf_t q_t(x)\ge \E  B_1\exp(-C|X_t(x)|^p)\ge B_1'\exp(-CL^p|x|^p).
\]

\end{proof}

\section{Some missing details in  Theorem \ref{thm:gradient_estimate} }\label{app:compute}

This appendix provides supplementary computational details for Theorem \ref{thm:gradient_estimate}, specifically documenting the algebraic expansions and simplifications required in the derivation process.
Recall that
\begin{gather*}
v:=\frac{|\nabla f|^2}{(1-f)^2}=\frac{f_k f_k}{(1-f)^2}, 
 \quad 
 f = \log \rho.
\end{gather*}
For $v$, one has
$$
\begin{aligned}
	\mathcal{L} v= &   v_{i i}+b^i v_i-v_t  \\
	= & \frac{2   f_{k i} f_{k i}}{(1-f)^2}+\frac{8   f_k f_{k i} f_i}{(1-f)^3}+\frac{6   f_k f_k f_i f_i}{(1-f)^4} 
	 -\frac{4   f_k f_{i k} f_i}{(1-f)^2}-\frac{2   f_i f_i f_k f_k}{(1-f)^3}-R_1,
\end{aligned}
$$
where
$$
R_1:=\frac{2 c f_k f_k}{(1-f)^3}+2 \frac{b_k^i f_k f_i+c_k f_k}{(1-f)^2}.
$$
By rearranging terms and introducing a parameter $\varepsilon $, we rewrite the expression as
$$
\begin{aligned}
	L v= & \frac{2(1-\varepsilon)   f_{k i} f_{k i}}{(1-f)^2}+\frac{2   f_i v_i}{1-f}-2   f_i v_i \\
	& +2\left[\frac{\varepsilon   f_{k i} f_{k i}}{(1-f)^2}+\frac{2   f_k f_{k i} f_i}{(1-f)^3}+\frac{\varepsilon^{-1}   f_k f_k f_i f_i}{(1-f)^4}\right] \\
	& +\left(2-2 \varepsilon^{-1}\right) \frac{  f_k f_k f_i f_i}{(1-f)^4}+\frac{2   f_i f_i f_k f_k}{(1-f)^3}-R_1,
\end{aligned}
$$
then taking $\varepsilon=2 / 3$,
$$
\begin{aligned}
\tilde{L} v :=	L v-\frac{2 f}{1-f}   f_i v_i & \geq \frac{2   f_{k i} f_{k i}}{3(1-f)^2}+\frac{  f_k f_k f_i f_i}{(1-f)^3}-R_1 \\
	& \geq \frac{2}{3}  \frac{  f_{k i} f_{k i}}{(1-f)^2}+   (1-f) v^2-R_1 .
\end{aligned}
$$
Based on Young inequality and the boundedness of $ \nabla U$ and its derivatives, we obtain:
$$
\begin{aligned}
\tilde{\mathcal{L}} v 
\geq 
\frac{  }{2}(1-f) v^2-M,
\end{aligned}
$$
where
$$
M = C_{\varepsilon} \left( 1 +\|\partial b\|_{L^{\infty}}^2 +\|c\|_{L^{\infty}}^2+\|\partial c\|_{L^{\infty}}^2\right).
$$
Recall that
\[
u=\frac{f_{k\ell}f_{k\ell}}{A(1-f)^2-f_j f_j}.
\]
By direct computation, one has
\begin{equation}
\begin{split}
	L u= &   u_{ii}+  b^i u_i-u_t\\
	= & \frac{2 f_{k\ell i}f_{k\ell i}}{A(1-f)^2-f_j f_j}
	+\frac{8 f_{k\ell}f_{k\ell i}}{(A(1-f)^2-f_jf_j)^2}(A(1-f)f_i+f_j f_{ji})\\
	 & +\frac{8 f_{k\ell}f_{k\ell}}{(A(1-f)^2-f_jf_j)^3}(A(1-f)f_i+f_j f_{ji})^2
	-\frac{4f_{k\ell}(f_{i\ell}f_{ki}+f_i f_{k\ell i})}{A(1-f)^2-f_jf_j}\\
	& +\frac{2f_{k\ell}f_{k\ell}}{(A(1-f)^2-f_jf_j)^2}[f_{ji}f_{ji}-Af_i^2-A(1-f)f_i^2-2f_jf_i f_{ij}]
	 - R_2,
\end{split}
\end{equation}
where
$$
R_2 := \frac{2f_{k\ell} (f_i b^i_{\ell k} + b^i_\ell f_{i k} + b^i_k f_{i\ell} + c_{kl})}{A(1-f)^2 - f_j f_j} 
+ \frac{2f_{k\ell} f_{k\ell} \left( A(1-f) c + f_j c_j \right)}{(A(1-f)^2 - f_j f_j)^2}.
$$
Noting that
\[
u_i=\frac{2f_{k\ell}f_{k\ell i}}{A(1-f)^2-f_jf_j}
+\frac{2f_{k\ell}f_{k\ell}}{(A(1-f)^2-f_jf_j)^2}(A(1-f)f_i+f_j f_{ji}),
\]
we find that
\begin{multline*}
	L u-4\frac{A(1-f)f_i+f_k f_{ki}}{A(1-f)^2-f_jf_j} u_i=
	\frac{2 f_{k\ell i}f_{k\ell i}}{A(1-f)^2-f_j f_j}
	-\frac{4f_{k\ell}(f_{i\ell}f_{ki}+f_i f_{k\ell i})}{A(1-f)^2-f_jf_j}\\
	+\frac{2f_{k\ell}f_{k\ell}}{(A(1-f)^2-f_jf_j)^2}[-Af_i^2-A(1-f)f_i^2-2f_jf_i f_{ij}]
	+\frac{2f_{k\ell}f_{k\ell}f_{ji}f_{ji}}{(A(1-f)^2-f_jf_j)^2} - R_2.
\end{multline*}
Applying Cauchy's inequality repeatedly,
$f_{k\ell}f_{i\ell}f_{ki}\le |f_{k\ell}f_{k\ell}|^{3/2}$
and thus
\begin{gather*}
	-\frac{4f_{k\ell}f_{i\ell}f_{ki}}{A(1-f)^2-f_jf_j}
	\ge -\epsilon u^2-4C\epsilon^{-1}(A(1-f)^2-f_jf_j)^{3/2}.
\end{gather*}
Moreover
\[
-\frac{4f_{k\ell}f_i f_{k\ell i}}{A(1-f)^2-f_jf_j}
\ge -\epsilon \frac{f_{k\ell i}f_{k\ell i}}{A(1-f)^2-f_jf_j}
-\epsilon u^2-C(\frac{f_if_i}{A(1-f)^2-f_jf_j})^2.
\]
Similarly,
\[
\frac{2f_{k\ell}f_{k\ell}}{(A(1-f)^2-f_jf_j)^2}[-A(2-f)f_i^2]
\ge -\epsilon u^2-C\frac{[(2-f) f_i^2]^2}{(A(1-f)^2-f_jf_j)^2},
\]
and
\[
\frac{2f_{k\ell}f_{k\ell}}{(A(1-f)^2-f_jf_j)^2}[-2f_jf_i f_{ij}]
\ge -\epsilon u^2-C\frac{(f_if_i)^3}{(A(1-f)^2-f_jf_j)^{3/2}}.
\]
It remains to justify how the remainder terms produced above 
(such as quartic or cubic powers of $\nabla f$ divided by the denominator, 
or positive powers of $A(1-f)^2-f_j f_j$) are absorbed into the control function $q(x)$. 
Recall that $q(x)=1+|x|^r$ with $r$ sufficiently large. 

As a representative example, consider
\[
\frac{(f_if_i)^3}{(A(1-f)^2-f_jf_j)^{3/2}},
\]
whose numerator grows at most cubically in $\nabla f$ 
while the denominator is strictly positive by construction. 
Hence this term has at most polynomial growth in $|x|$ 
and is bounded by $Cq(x)$. 
Similarly, 
\[
\Big(\frac{f_if_i}{A(1-f)^2-f_jf_j}\Big)^2 \le C, 
\qquad 
\frac{[(2-f) f_i^2]^2}{(A(1-f)^2-f_jf_j)^2} \le Cq(x),
\]
and likewise for the positive term $(A(1-f)^2-f_jf_j)^{3/2}$. 

Therefore, all remainder terms can be uniformly controlled by $Cq(x)$, 
and the final inequality reads
\[
\tilde{\cL}u := \cL u-4\frac{A(1-f)f_i+f_k f_{ki}}{A(1-f)^2-f_jf_j} u_i
\;\ge\; \frac{ f_{k\ell i}f_{k\ell i}}{A(1-f)^2-f_j f_j}
+ \tfrac{3}{2}u^2 - C q(x) - R_2(x),
\]
as used in the main text.

\section{Calculation and pointwise estimates of the remainder terms}\label{app:pointremainder}

In this subsection, we aim to find some pointwise estimates of $E_i(t,x)$ defined above.

\subsection{The calculation of the formulas}

Consider the term for $I_1$. We split
\begin{multline*}
\frac{1}{2}\int (\log \frac{\bar{\rho}}{\rho}) I_1\,dx  = \frac{1}{2} \beta^{-1} \int (\log\frac{\bar{\rho}}{\rho})
	 (\Delta h_2) (s, \phi_{-s} (x) )  J(-s, x )  dx  \\
	 -\frac{1}{2} \beta^{-1}\int (\log\frac{\bar{\rho}}{\rho})
	 \Delta_x \big( h_2(s, \phi_{-s} (x) ) J(-s, x ) \big) dx=:I_{11}+I_{12}.
\end{multline*}
For the first term, by change of variables
\[
u=\phi_{-s}(x) \Rightarrow x=\phi_s(u),
\]
it reduces to
\begin{gather*}
\begin{split}
I_{11}&=\frac{1}{2}\beta^{-1}\int (\log\frac{\bar{\rho}(t, \phi_{s} (u))}{\rho(t,\phi_{s} (u))})
	 \Delta_u h_2(s, u) du\\
	&= -\frac{1}{2}\beta^{-1}\int (\nabla_u\phi_s(u)\cdot \nabla \log\frac{\bar{\rho}(t,\phi_{s} (u))}{\rho(t,\phi_{s} (u))})
	\cdot \nabla_u h_2(s, u) du\\
	&=-\frac{1}{2}\beta^{-1}\int \left((\nabla_x\phi_{-s}(x))^{-1}\cdot \nabla \log\frac{\bar{\rho}(t, x)}{\rho(t,x)}\right)
	\cdot (\nabla h_2)(s, \phi_{-s}(x)) J(-s, x)dx.
\end{split}
\end{gather*}

By integration by parts, the second term is then given by
\begin{multline*}
I_{12}=\frac{1}{2}\int \nabla \log\frac{\bar{\rho}(t, x)}{\rho(t,x)}
\cdot\left(\nabla_x\phi_{-s}(x)\cdot \nabla h_2(s, \phi_{-s}(x)) J(-s, x)+h_2(s, \phi_{-s}(x) \nabla J(-s, x)) \right).
\end{multline*}
Hence, one has
\begin{multline}
\frac{1}{2}\int (\log \frac{\bar{\rho}}{\rho}) I_1\,dx =  \frac{1}{2} \beta^{-1} \int  \nabla \log\frac{\bar{\rho}}{\rho}
	 \cdot \Big[\nabla_x\phi_{-s}(x)-(\nabla_x\phi_{-s}(x))^{-T}\Big]\cdot\nabla h_2(s, \phi_{-s} (x) )
	 J(-s,x )\,dx  \\
	   +\frac{1}{2}\beta^{-1}\int \nabla \log\frac{\bar{\rho}}{\rho}\cdot \nabla J(-s, x) h_2(s, \phi_{-s}(x))\,dx.
\end{multline}

By integration by parts, it is relatively straightforward to see that
\begin{equation}
\begin{split}
\frac{1}{2}\int (\log \frac{\bar{\rho}}{\rho}) I_2\,dx &= \frac{1}{2}\int \nabla \log\frac{\bar{\rho}(t,x)}{\rho(t,x)}\cdot  \int \frac{1}{(4 \pi \beta^{-1} s)^{\frac{d}{2}}}
		e^{-\frac{\beta |x-y|^2}{4 s} } \big( b(y) - b(x) \big) 
		h_1(s, y) dy dx\\
	&=\frac{1}{2}\int \nabla \log\frac{\bar{\rho}}{\rho}\cdot  \int 
		(y-x)\cdot \overline{\nabla b}(x, y)  \frac{e^{-\frac{\beta |x-y|^2}{4 s} }}{(4 \pi \beta^{-1} s)^{\frac{d}{2}}}  h_1(s, y) dy dx,
\end{split}
\end{equation}
where we recall
\begin{gather}
\overline{\nabla b}(x, y):=\int_0^1\nabla b(\lambda y+(1-\lambda)x)d\lambda.
\end{gather}

Formal expansion of the semigroup indicates that this should agree with $R_1$ in the $O(s)$ order. We need to do some treatment to change this form here into a form that matches the expressions in $R_1$.
To this end, we split the above expression into two halves.
Noting that 
\[
\frac{1}{2}(y-x)e^{-\frac{\beta |x-y|^2}{4 s} }
=-s\beta^{-1}\nabla_y e^{-\frac{\beta |x-y|^2}{4 s} }=s\beta^{-1}\nabla_xe^{-\frac{\beta |x-y|^2}{4 s} },
\]
one finds that
\begin{equation}
\begin{split}
&\frac{1}{2}\int (\log \frac{\bar{\rho}}{\rho}) I_1\,dx =
\frac{s\beta^{-1}}{2}\int \nabla \log\frac{\bar{\rho}}{\rho}\cdot  \Big[\int 
		\nabla_x \frac{e^{-\frac{\beta |x-y|^2}{4 s} }}{(4 \pi \beta^{-1} s)^{\frac{d}{2}}} \cdot \overline{\nabla b}(x, y)    h_1(s, y) dy\Big] dx\\
&-\frac{s\beta^{-1}}{2}\int \nabla \log\frac{\bar{\rho}}{\rho}\cdot \Big[ \int 
		\nabla_y \frac{e^{-\frac{\beta |x-y|^2}{4 s} }}{(4 \pi \beta^{-1} s)^{\frac{d}{2}}} \cdot \overline{\nabla b}(x, y)    h_1(s, y) dy\Big] dx\\
&=:I_{21}+I_{22}.
\end{split}
\end{equation}
Integration by parts, one finds that
\begin{multline*}
I_{21}=-\frac{s\beta^{-1}}{2}\int \nabla^2 \log\frac{\bar{\rho}}{\rho}: \int 
		\frac{e^{-\frac{\beta |x-y|^2}{4 s} }}{(4 \pi \beta^{-1} s)^{\frac{d}{2}}}  \overline{\nabla b}(x, y)   h_1(s, y) dy dx\\
-\frac{s\beta^{-1}}{2}\int \nabla \log\frac{\bar{\rho}}{\rho}\cdot \int 
		\frac{e^{-\frac{\beta |x-y|^2}{4 s} }}{(4 \pi \beta^{-1} s)^{\frac{d}{2}}}  \int_0^1 (1-\lambda) \Delta b(\lambda y+(1-\lambda)x)d\lambda    h_1(s, y) dy dx.
\end{multline*}
We do integration by parts again for the first term in $I_{21}$, but for the index that taking contraction with the components of $b$, to have
\begin{multline*}
I_{21}=\frac{s\beta^{-1}}{2}\int \nabla \log\frac{\bar{\rho}}{\rho}\cdot \int \frac{e^{-\frac{\beta |x-y|^2}{4 s} }}{(4 \pi \beta^{-1} s)^{\frac{d}{2}}}  \int_0^1 (1-\lambda)\nabla (\nabla\cdot b)(\lambda y+(1-\lambda)x)d\lambda    h_1(s, y) dy dx\\
+\frac{s\beta^{-1}}{2}\int \nabla \log\frac{\bar{\rho}}{\rho}\cdot \Big[\int  \overline{\nabla b}(x, y) 
\cdot \nabla_x\frac{e^{-\frac{\beta |x-y|^2}{4 s} }}{(4 \pi \beta^{-1} s)^{\frac{d}{2}}}  h_1(s, y) dy\Big] dx\\
-\frac{s\beta^{-1}}{2}\int \nabla \log\frac{\bar{\rho}}{\rho}\cdot \int  \frac{e^{-\frac{\beta |x-y|^2}{4 s} }}{(4 \pi \beta^{-1} s)^{\frac{d}{2}}}  \int_0^1 (1-\lambda) \Delta b(\lambda y+(1-\lambda)x)d\lambda    h_1(s, y) dy dx.
\end{multline*}
Then, we change $\nabla_x$ in the second term above to $-\nabla_y$ and do integration by parts and have
\begin{multline*}
I_{21}=\frac{s\beta^{-1}}{2}\int \nabla \log\frac{\bar{\rho}}{\rho}\cdot \int \frac{e^{-\frac{\beta |x-y|^2}{4 s} }}{(4 \pi \beta^{-1} s)^{\frac{d}{2}}}  \overline{\nabla(\nabla\cdot b)}(x, y)    h_1(s, y) dy dx\\
+\frac{s\beta^{-1}}{2}\int \nabla \log\frac{\bar{\rho}}{\rho}\cdot \Big[\int  \overline{\nabla b}(x, y)
\frac{e^{-\frac{\beta |x-y|^2}{4 s} }}{(4 \pi \beta^{-1} s)^{\frac{d}{2}}}\cdot  \nabla h_1(s, y) dy\Big] dx\\
-\frac{s\beta^{-1}}{2}\int \nabla \log\frac{\bar{\rho}}{\rho}\cdot \int 
		\frac{e^{-\frac{\beta |x-y|^2}{4 s} }}{(4 \pi \beta^{-1} s)^{\frac{d}{2}}}  \int_0^1 (1-\lambda) \Delta b(\lambda y+(1-\lambda)x)d\lambda    h_1(s, y) dy dx,
\end{multline*}
where similarly
\begin{gather}
\overline{\nabla(\nabla\cdot b)}(x, y):=\int_0^1 \nabla (\nabla\cdot b)(\lambda y+(1-\lambda)x)d\lambda.
\end{gather}

Direct integration by parts for $I_{22}$ gives
\begin{multline*}
I_{22}=\frac{s\beta^{-1}}{2}\int \nabla \log\frac{\bar{\rho}}{\rho}\cdot \int 
		\frac{e^{-\frac{\beta |x-y|^2}{4 s} }}{(4 \pi \beta^{-1} s)^{\frac{d}{2}}}  \int_0^1 \lambda \Delta b(\lambda y+(1-\lambda)x)d\lambda    h_1(s, y) dy dx\\
+\frac{s\beta^{-1}}{2}\int \nabla \log\frac{\bar{\rho}}{\rho}\cdot \int 
		\frac{e^{-\frac{\beta |x-y|^2}{4 s} }}{(4 \pi \beta^{-1} s)^{\frac{d}{2}}}  \overline{\nabla b}(x, y)^T \cdot \nabla  h_1(s, y) dy dx.
\end{multline*}

Adding all the above up, we arrive at the expressions in \eqref{eq:remainderformula}.

\subsection{Estimates of $E_1$}

By the formulas above, one has $E_1=E_{11}+E_{12}$, where the main term $E_{11}$ is
\begin{multline*}
E_{11}(s,x)=\frac{\nabla_x\phi_{-s}(x)-(\nabla_x\phi_{-s}(x))^{-T}}{s}
\cdot \left(\int \frac{\exp(-\frac{\beta|\phi_{-s}(x)-y|^2}{4s})}{(4\pi \beta^{-1}s)^{d/2}}\nabla \bar{\rho}_n(y)\,dy\right)\\
+\int \frac{e^{-\frac{\beta |x-y|^2}{4 s} }}{(4 \pi \beta^{-1} s)^{\frac{d}{2}}}   [\overline{\nabla b}(x,y)+\overline{\nabla b}(x,y)^T] \cdot \nabla\bar{\rho}_n(\phi_{-s}(y))J(-s, y) dy,
\end{multline*}
while the extra remainder is 
\begin{multline*}
E_{12}=\frac{\nabla_x\phi_{-s}(x)-(\nabla_x\phi_{-s}(x))^{-T}}{s}
\cdot \left(\int \frac{\exp(-\frac{\beta|\phi_{-s}(x)-y|^2}{4s})}{(4\pi \beta^{-1}s)^{d/2}}\nabla \bar{\rho}_n(y)\,dy\right)(J(-s, x)-1)\\
+\int \frac{e^{-\frac{\beta |x-y|^2}{4 s} }}{(4 \pi \beta^{-1} s)^{\frac{d}{2}}}   [\overline{\nabla b}(x,y)+\overline{\nabla b}(x,y)^T] \cdot (\nabla\phi_{-s}(y)-I)\cdot \nabla\bar{\rho}_n(\phi_{-s}(y))J(-s, y) dy\\
+\int \frac{e^{-\frac{\beta |x-y|^2}{4 s} }}{(4 \pi \beta^{-1} s)^{\frac{d}{2}}}   [\overline{\nabla b}(x,y)+\overline{\nabla b}(x,y)^T] \cdot  \bar{\rho}_n(\phi_{-s}(y))\nabla J(-s, y) dy.
\end{multline*}

\noindent {\bf Estimate of the remainder} 

Since $\partial_i J/J=\mathrm{tr}(\nabla \phi^{-1} \partial_i \nabla \phi)$, it is easy to see by the assumptions on $b$ that
\[
\|\nabla J/J\|_{\infty}\le Cs.
\]
where $C$ is uniformly bounded if we require $|s|\le 1$.
Hence, the last term in $E_{12}$ is simply controlled by $Ce^{s\cL_{\xi(2)}}e^{s\cL_{\xi(1)}}\bar{\rho_n} s$.

Moreover, for $|s|\le 1$, it is also easy to see that
\[
|\nabla \phi_{-s}(y)-I|\le C |s|, \quad |J(-s, x)-1|\le Cs, |\frac{\nabla_x\phi_{-s}(x)-(\nabla_x\phi_{-s}(x))^{-T}}{s}|\le C.
\]
 Hence, the first two terms are controlled by
\begin{multline*}
|E_{12}(s,x)| \le Cs\int \frac{e^{-\frac{\beta |\phi_{-s}(x)-y|^2}{4 s} }}{(4 \pi \beta^{-1} s)^{\frac{d}{2}}}  |\nabla\log\bar{\rho}_n(y)|\bar{\rho}_n(y)\,dy\\
+Cs\int \frac{e^{-\frac{\beta |x-y|^2}{4 s} }}{(4 \pi \beta^{-1} s)^{\frac{d}{2}}}  |\nabla\log\bar{\rho}_n(\phi_{-s}(y))|\bar{\rho}_n(\phi_{-s}(y))J(-s, y) dy.
\end{multline*}
These two terms are similar. We take the latter one as the example to show how to control. 
By Lemma \ref{lem:flowmap} and Theorem \ref{thm:deriva_density}, we find that (for $|s|\le 1$)
\[
|\nabla\log\bar{\rho}_n(\phi_{-s}(y))|\le C(1+|\phi_{-s}(y)|^p)
\le C(1+|y|^p+|b(x)|^p|s|^p)\le C(1+|x|^p+|x-y|^p),
\]
where the concrete value of $C$ has been changed. 
Hence, 
\begin{multline*}
\int \frac{e^{-\frac{\beta |x-y|^2}{4 s} }}{(4 \pi \beta^{-1} s)^{\frac{d}{2}}}  |\nabla\log\bar{\rho}_n(\phi_{-s}(y))|\bar{\rho}_n(\phi_{-s}(y))J(-s, y) dy\le \\
 C(1+|x|^p)e^{s\cL_{\xi(2)}}e^{s\cL_{\xi(1)}}\bar{\rho}_n(x)
+\int \frac{e^{-\frac{\beta |x-y|^2}{4 s} }}{(4 \pi \beta^{-1} s)^{\frac{d}{2}}}  |x-y|^p \bar{\rho}_n(\phi_{-s}(y))J(-s, y) dy
\end{multline*}
For the latter, take
\[
\bar{R} = \beta^{-1/2} \sqrt{8s} \max\left\{ 1, \sqrt{ |\log \bar{\rho}|} \right\}.
\]
One then has
\[
\int_{|x-y|\le \bar{R}}
\frac{e^{-\frac{\beta |x-y|^2}{4 s} }}{(4 \pi \beta^{-1} s)^{\frac{d}{2}}}  |x-y|^p \bar{\rho}_n(\phi_{-s}(y))J(-s, y) dy
\le \bar{R}^p e^{s\cL_{\xi(2)}}e^{s\cL_{\xi(1)}}\bar{\rho}_n(x)
\]
while for $|x-y|\ge \bar{R}$, $\exp(-\frac{\beta|x-y|^2}{8s})|x-y|^p\le C_p(s/\beta)^{p/2}\exp(-\frac{\beta|x-y|^2}{16s})$
for some $C_p$ independent of $s$ and $x, y$ so that
\begin{gather}\label{eqaux:largedistance}
\begin{split}
& \int_{|x-y|\ge \bar{R}} 
\frac{e^{-\frac{\beta |x-y|^2}{4 s} }}{(4 \pi \beta^{-1} s)^{\frac{d}{2}}}  |x-y|^p \bar{\rho}_n(\phi_{-s}(y))J(-s, y) dy\\
\le & C(s/\beta)^{p/2}\exp(-\frac{\beta \bar{R}^2}{8s})\int \frac{e^{-\frac{\beta |x-y|^2}{ 16s} }}{(4 \pi \beta^{-1} s)^{\frac{d}{2}}} \bar{\rho}_n(\phi_{-s}(y))J(-s, y) dy.
\end{split}
\end{gather}
Noting that
\[
\exp(-\frac{\beta \bar{R}^2}{8s})\le \exp(-|\log \bar{\rho}|)\le \min\{1, \bar{\rho}\},
\]
and the fact that $\bar{\rho}_n$ is bounded for $t\le T$, we find that
\begin{multline*}
|E_{12}(s,x)|\le Cs(1+|x|^{p})e^{s\cL_{\xi(2)}}e^{s\cL_{\xi(1)}}\bar{\rho_n} 
+Cs(s/\beta)^{p/2} \max\{1, |\log \bar{\rho}|^{p/2}\}e^{s\cL_{\xi(2)}}e^{s\cL_{\xi(1)}}\bar{\rho_n}\\
+Cs(s/\beta)^{p/2}\min\{1, \bar{\rho}(t,x)\}
\le Cs(1+|x|^{p'})\bar{\rho}(t, x),
\end{multline*}
where we used that $|\log \bar{\rho}|$ has polynomial growth.

\vskip 0.1 in

\noindent {\bf Estimate of the main term}

To estimate the main terms in $E_{11}$, we first do change of variables $y'=\phi_{-s}(y)$ and split $E_{11}:=M_1+M_2$, with 
\begin{multline*}
M_1=\frac{\nabla_x\phi_{-s}(x)-(\nabla_x\phi_{-s}(x))^{-T}}{s}
\cdot \\
\int \frac{1}{(4\pi \beta^{-1}s)^{d/2}}\left[\exp(-\frac{\beta|\phi_{-s}(x)-y|^2}{4s})-\exp(-\frac{\beta |x-\phi_{s}(y)|^2}{4 s} )\right]\nabla \bar{\rho}_n(y)\,dy
\end{multline*}
and
\begin{multline*}
M_2=
\int \frac{e^{-\frac{\beta |x-\phi_{s}(y)|^2}{4 s} }}{(4 \pi \beta^{-1} s)^{\frac{d}{2}}}  \\
\left( [\overline{\nabla b}(x,\phi_s(y))+\overline{\nabla b}(x,\phi_s(y))^T]+\frac{\nabla_x\phi_{-s}(x)-(\nabla_x\phi_{-s}(x))^{-T}}{s} \right)\cdot \nabla\bar{\rho}_n(y)dy
\end{multline*}

For $M_1$,  by Lemma \ref{lem:flowmap}, we find that
\[
||x-\phi_{s}(y)|^2-|\phi_{-s}(x)-y|^2|\le C |x-\phi_s(y)|^2 s.
\]
Since
\begin{multline*}
\exp(-\frac{\beta|\phi_{-s}(x)-y|^2}{4s})-\exp(-\frac{\beta |x-\phi_{s}(y)|^2}{4 s} )\\
=\exp(-\frac{\beta |x-\phi_{s}(y)|^2}{4 s} ) \left[\exp(\frac{\beta[ |x-\phi_{s}(y)|^2-|\phi_{-s}(x)-y|^2]}{4s})-1 \right]\\
=\exp(-\frac{\beta|\phi_{-s}(x)-y|^2}{4s})(1-\exp(\frac{\beta[|\phi_{-s}(x)-y|^2- |x-\phi_{s}(y)|^2]}{4 s} )).
\end{multline*}
If we simply bound 
\[
\exp(\frac{\beta[ |x-\phi_{s}(y)|^2-|\phi_{-s}(x)-y|^2]}{4s})-1
\le C\exp(C|x-\phi_s(y)|^2)|x-\phi_s(y)|^2,
\]
we would have trouble to deal with the extra exponential term here. To resolve this, we note that $|e^u-1|\le |u|$ for $u<1$. Hence, we may choose the different bounds for different cases. In particular,   if $|x-\phi_{s}(y)|^2-|\phi_{-s}(x)-y|^2\le 0$, we use the bound
\[
|\exp(-\frac{\beta|\phi_{-s}(x)-y|^2}{4s})-\exp(-\frac{\beta |x-\phi_{s}(y)|^2}{4 s} )|
\le C\exp(-\frac{\beta |x-\phi_{s}(y)|^2}{4 s} ) |x-\phi_s(y)|^2
\]
Otherwise, we use the bound
\[
C\exp\left(-\frac{\beta|\phi_{-s}(x)-y|^2}{4s}\right) |x-\phi_s(y)|^2.
\]
Hence, we find 
\begin{multline*}
M_1\le C\int \frac{1}{(4\pi \beta^{-1}s)^{d/2}}\left[\exp\left(-\frac{\beta |x-\phi_{s}(y)|^2}{4 s} \right)+
\exp\left(-\frac{\beta|\phi_{-s}(x)-y|^2}{4s}\right) \right]\\
|x-\phi_s(y)|^2|\nabla \log\bar{\rho}_n(y)|\bar{\rho}_n(y)\,dy
\end{multline*}
Similar as above, we use the bound for $s\le 1$:
\[
|\nabla \log\bar{\rho}_n(y)|\le C(1+|y|^p)\le C(1+|\phi_s(y)|^p)\le C(1+|x|^p+|x-\phi_s(y)|^p).
\]

For $|x-\phi_s(y)|\le \bar{R}$, we find
\begin{multline*}
\int_{|x-\phi_s(y)|\le \bar{R}}
\frac{|x-\phi_s(y)|^2|\nabla \log\bar{\rho}_n(y)|\bar{\rho}_n(y)}{(4\pi \beta^{-1}s)^{d/2}}\left[\exp\left(-\frac{\beta |x-\phi_{s}(y)|^2}{4 s} \right)+
\exp\left(-\frac{\beta|\phi_{-s}(x)-y|^2}{4s}\right) \right]
\,dy\\ \le C\beta^{-1}s[(1+|x|^p) (1+|\log\bar{\rho}|)+(1+|\log\bar{\rho}|)^{p/2}][e^{s\cL_{\xi(1)}}e^{s\cL_{\xi(2)}}
+e^{s\cL_{\xi(2)}}e^{s\cL_{\xi(1)}}]\bar{\rho}_n\\
\le C s(1+|x|^{p'})\bar{\rho}(t, x),
\end{multline*}
where we used $[e^{s\cL_{\xi(1)}}e^{s\cL_{\xi(2)}}
+e^{s\cL_{\xi(2)}}e^{s\cL_{\xi(1)}}]\bar{\rho}_n=2\bar{\rho}$ and that $|\log \bar{\rho}|$ has polynomial growth.

For $|x-\phi_s(y)|\ge \bar{R}$, the estimate is similar to \eqref{eqaux:largedistance}
and we find
\begin{equation*}
    \begin{split}
        &\int_{|x-\phi_s(y)|\ge \bar{R}}
      \frac{|x-\phi_s(y)|^2|\nabla \log\bar{\rho}_n(y)|\bar{\rho}_n(y)}{(4\pi \beta^{-1}s)^{d/2}}\left[\exp\left(-\frac{\beta |x-\phi_{s}(y)|^2}{4 s} \right)+
\exp\left(-\frac{\beta|\phi_{-s}(x)-y|^2}{4s}\right) \right]
\,dy \\  
\le & Cs\min\{1, \bar{\rho}(t, x)\}(1+|x|^p).
    \end{split}
\end{equation*}

For $M_2$, we first note that
$\nabla_x\phi_{-s}(x)^{-T}=\nabla\phi_s(\phi_{-s}(x))^T$, and thus
\[
\frac{\nabla_x\phi_{-s}(x)-(\nabla_x\phi_{-s}(x))^{-T}}{s}
=-\nabla b(x)-\nabla b(x)^T+B(s,x) s.
\]
with $|B(s,x)|\le C(1+|x|)$. The reminder term in $M_2$ is treated similarly as before.
The main term becomes
\begin{multline*}
\tilde{M}_{2}=
\int \frac{e^{-\frac{\beta |x-\phi_{s}(y)|^2}{4 s} }}{(4 \pi \beta^{-1} s)^{\frac{d}{2}}} 
\left( [\overline{\nabla b}(x,\phi_s(y))+\overline{\nabla b}(x,\phi_s(y))^T]-\nabla b(x)-\nabla b(x)^T \right)\cdot \nabla\bar{\rho}_n(y)dy
\end{multline*}
The difference will basically give $(y-x)$ which makes the $y$ derivative of the Gaussian kernel times $\beta^{-1}s$. Integration by parts, we find that
\begin{gather*}
\tilde{M}_2
\le C(\|\Delta \nabla b\|_{\infty}+|\nabla^2b|_{\infty})s\int \frac{e^{-\frac{\beta |x-\phi_{s}(y)|^2}{4 s} }}{(4 \pi \beta^{-1} s)^{\frac{d}{2}}} (|\nabla\bar{\rho}_n(y)|+
|\nabla^2 \bar{\rho}_n(y)|)dy.
\end{gather*}
Using the polynomial growth of $|\nabla \log\bar{\rho}_n|$
and $|\nabla^2\log \bar{\rho}_n|$ and with the formula
\[
\nabla^2\bar{\rho}_n=(\nabla^2\log \bar{\rho}_n+\nabla\bar{\rho}_n^{\otimes 2})\bar{\rho}_n,
\]
one can treat this using the same approach above. Hence
\[
\tilde{M}_2 \le C s ( 1+ |x|p') \bar{\rho}(t,x).
\]

\subsection{Estimates of $E_2$ and $E_3$}

Using the formulas of $h_i$, we have
\begin{multline}
E_2(s,x)=s^{-1}\nabla J(-s, x) \int \frac{1}{(4\pi \beta^{-1}s)^{d/2}}\exp(-\frac{\beta|\phi_{-s}(x)-y|^2}{4s}) \bar{\rho}_n(y)\,dy\\
+
\int \frac{e^{-\frac{\beta |x-y|^2}{4 s} }}{(4 \pi \beta^{-1} s)^{\frac{d}{2}}}  \overline{\nabla (\nabla\cdot b)}(x,y)  \bar{\rho}_n(\phi_{-s}(y)) J(-s, y) dy,
\end{multline}
Using the equation for $J$, one has
\begin{gather}
\partial_t\nabla J(t,x)= \nabla\phi_t(x)\cdot \nabla \mathrm{div}(b) J(t,x)+\mathrm{div}(b)(\phi_t(x))\nabla J(t,x).
\end{gather}
Note that $| \nabla\phi_t(x)-I |\le C|t|$, and $\nabla J(0, x)=0$. 
Then, the estimate of $E_2$ is similar to $E_1$ but simpler as $h_1$ and $h_2$ are used instead of their gradients are used.
Here, we would use only the polynomial growth of first order derivatives of $ \log \bar{\rho}$.

The $E_3$ term is relatively easy. In fact, 
\begin{multline}
E_3(s,x)=\int \frac{e^{-\frac{\beta |x-y|^2}{4 s} }}{(4 \pi \beta^{-1} s)^{\frac{d}{2}}}  \int_0^1 \lambda [\Delta b(\lambda y+(1-\lambda)x)
-\Delta b(\lambda x+(1-\lambda)y)]d\lambda    h_1(s, y) dy
\end{multline}
The difference will gives $(y-x)$ and then we make this as the $y$ derivative of the Gaussian kernel. Integration by parts,
we need to control on $\nabla h_1(s, y)$. Hence, as long as the fourth order derivative of $b$ is bounded, this term is controlled similarly 
as $E_{12}$ above.

\bibliographystyle{plain}
\bibliography{refer}

\begin{thebibliography}{10}

\bibitem{abdulle2015long}
Assyr Abdulle, Gilles Vilmart, and Konstantinos~C Zygalakis.
\newblock Long time accuracy of {L}ie--{T}rotter splitting methods for
  {L}angevin dynamics.
\newblock {\em SIAM Journal on Numerical Analysis}, 53(1):1--16, 2015.

\bibitem{allen2017computer}
Michael~P Allen and Dominic~J Tildesley.
\newblock {\em Computer simulation of liquids}.
\newblock Oxford university press, 2017.

\bibitem{anderson2012multilevel}
David~F Anderson and Desmond~J Higham.
\newblock Multilevel {M}onte {C}arlo for continuous time {M}arkov chains, with
  applications in biochemical kinetics.
\newblock {\em Multiscale Modeling \& Simulation}, 10(1):146--179, 2012.

\bibitem{andrieu2003introduction}
Christophe Andrieu, Nando De~Freitas, Arnaud Doucet, and Michael~I Jordan.
\newblock An introduction to {MCMC} for machine learning.
\newblock {\em Machine learning}, 50(1):5--43, 2003.

\bibitem{atchade2006adaptive}
Yves~F Atchad{\'e}.
\newblock An adaptive version for the {M}etropolis adjusted {L}angevin
  algorithm with a truncated drift.
\newblock {\em Methodology and Computing in applied Probability},
  8(2):235--254, 2006.

\bibitem{bernton2018langevin}
Espen Bernton.
\newblock Langevin {M}onte {C}arlo and {JKO} splitting.
\newblock In {\em Conference on learning theory}, pages 1777--1798. PMLR, 2018.

\bibitem{bou2020coupling}
Nawaf Bou-Rabee, Andreas Eberle, and Raphael Zimmer.
\newblock Coupling and convergence for {H}amiltonian {M}onte {C}arlo.
\newblock {\em The Annals of applied probability}, 30(3):1209--1250, 2020.

\bibitem{cheng2018convergence}
Xiang Cheng and Peter Bartlett.
\newblock Convergence of {L}angevin {MCMC} in {KL}-divergence.
\newblock In {\em Algorithmic learning theory}, pages 186--211. PMLR, 2018.

\bibitem{chewi2021optimal}
Sinho Chewi, Chen Lu, Kwangjun Ahn, Xiang Cheng, Thibaut Le~Gouic, and Philippe
  Rigollet.
\newblock Optimal dimension dependence of the {M}etropolis-adjusted {L}angevin
  algorithm.
\newblock In {\em Conference on Learning Theory}, pages 1260--1300. PMLR, 2021.

\bibitem{cho2024doubling}
Chien-Hung Cho, Dominic~W Berry, and Min-Hsiu Hsieh.
\newblock Doubling the order of approximation via the randomized product
  formula.
\newblock {\em Physical Review A}, 109(6):062431, 2024.

\bibitem{thomas2006elements}
Thomas~M Cover and Joy~A Thomas.
\newblock {\em Elements of information theory}.
\newblock Wiley-Interscience, Hoboken, New Jersey, 2nd edition, 2006.

\bibitem{dalalyan2019user}
Arnak~S Dalalyan and Avetik Karagulyan.
\newblock User-friendly guarantees for the {L}angevin {M}onte {C}arlo with
  inaccurate gradient.
\newblock {\em Stochastic Processes and their Applications},
  129(12):5278--5311, 2019.

\bibitem{durmus2015quantitative}
Alain Durmus and {\'E}ric Moulines.
\newblock Quantitative bounds of convergence for geometrically ergodic {M}arkov
  chain in the {W}asserstein distance with application to the {M}etropolis
  adjusted {L}angevin algorithm.
\newblock {\em Statistics and Computing}, 25(1):5--19, 2015.

\bibitem{durmus2019nonasymptotic}
Alain Durmus and {\'E}ric Moulines.
\newblock High-dimensional bayesian inference via the unadjusted langevin
  algorithm.
\newblock {\em Journal of Machine Learning Research}, 20(73):1--46, 2019.

\bibitem{durmus2019high}
Alain Durmus and Eric Moulines.
\newblock High-dimensional {B}ayesian inference via the unadjusted {L}angevin
  algorithm.
\newblock {\em Bernoulli}, 25(4A):2854--2882, 2019.

\bibitem{eberle2011reflection}
Andreas Eberle.
\newblock Reflection coupling and {W}asserstein contractivity without
  convexity.
\newblock {\em Comptes Rendus Mathematique}, 349(19-20):1101--1104, 2011.

\bibitem{eberle2016reflection}
Andreas Eberle.
\newblock Reflection couplings and contraction rates for diffusions.
\newblock {\em Probability theory and related fields}, 166(3):851--886, 2016.

\bibitem{eberle2019couplings}
Andreas Eberle, Arnaud Guillin, and Raphael Zimmer.
\newblock Couplings and quantitative contraction rates for {L}angevin dynamics.
\newblock {\em The Annals of Probability}, 47(4):1982--2010, 2019.

\bibitem{eberle2019quantitative}
Andreas Eberle, Arnaud Guillin, and Raphael Zimmer.
\newblock Quantitative {H}arris-type theorems for diffusions and
  {M}ckean--{V}lasov processes.
\newblock {\em Transactions of the American Mathematical Society},
  371(10):7135--7173, 2019.

\bibitem{frenkel2023understanding}
Daan Frenkel and Berend Smit.
\newblock {\em Understanding molecular simulation: from algorithms to
  applications}.
\newblock Elsevier, 2023.

\bibitem{gelman1995bayesian}
Andrew Gelman, John~B Carlin, Hal~S Stern, and Donald~B Rubin.
\newblock {\em Bayesian data analysis}.
\newblock Chapman and Hall/CRC, 1995.

\bibitem{hartmann2016molecular}
Carsten Hartmann.
\newblock Molecular dynamics. with deterministic and stochastic numerical
  methods, 2016.

\bibitem{hastings1970monte}
W~Keith Hastings.
\newblock Monte {C}arlo sampling methods using {M}arkov chains and their
  applications.
\newblock {\em Biometrika}, 57(1):97--109, 1970.

\bibitem{jin2020random}
Shi Jin, Lei Li, and Jian-Guo Liu.
\newblock Random batch methods ({RBM}) for interacting particle systems.
\newblock {\em Journal of Computational Physics}, 400:108877, 2020.

\bibitem{jing2024machine}
Yang Jing, Jiaheng Chen, Lei Li, and Jianfeng Lu.
\newblock A machine learning framework for geodesics under spherical
  {W}asserstein--{F}isher--{R}ao metric and its application for weighted sample
  generation.
\newblock {\em Journal of Scientific Computing}, 98(1):5, 2024.

\bibitem{kushner2003stochastic}
Harold~J Kushner and G~George Yin.
\newblock {\em Stochastic approximation and recursive algorithms and
  applications}, volume~35 of {\em Stochastic Modelling and Applied
  Probability}.
\newblock Springer, 2nd edition, 2003.

\bibitem{leimkuhler2013rational}
Benedict Leimkuhler and Charles Matthews.
\newblock Rational construction of stochastic numerical methods for molecular
  sampling.
\newblock {\em Applied Mathematics Research eXpress}, 2013(1):34--56, 2013.

\bibitem{leimkuhler2015molecular}
Benedict Leimkuhler and Charles Matthews.
\newblock {\em Molecular Dynamics: With Deterministic and Stochastic Numerical
  Methods}, volume~39 of {\em Interdisciplinary Applied Mathematics}.
\newblock Springer, 2015.

\bibitem{li2024geometric}
Lei Li, Jian-Guo Liu, and Yuliang Wang.
\newblock Geometric ergodicity of {SGLD} via reflection coupling.
\newblock {\em Stochastics and Dynamics}, 24(05):2450035, 2024.

\bibitem{li2023splitting}
Lei Li, Lin Liu, and Yuzhou Peng.
\newblock A splitting {H}amiltonian {M}onte {C}arlo method for efficient
  sampling.
\newblock {\em CSIAM Transactions on Applied Mathematics}, 4(1):41--73, 2023.

\bibitem{li2025generalized}
Lei Li, Xiangxian Luo, and Yinchen Luo.
\newblock A generalized discontinuous {H}amilton {M}onte {C}arlo for
  transdimensional sampling.
\newblock {\em arXiv preprint arXiv:2505.10108}, 2025.

\bibitem{li2025second}
Lei Li and Yuzhou Peng.
\newblock A second order {L}angevin sampler preserving positive volume for
  isothermal isobaric ensemble.
\newblock {\em Communications in Computational Physics}, 38(3):630--660, 2025.

\bibitem{li2025convergence}
Lei Li and Chen Wang.
\newblock Convergence of random splitting method for the {A}llen-{C}ahn
  equation in a background flow.
\newblock {\em Numerical Methods for Partial Differential Equations}.

\bibitem{li2024estimates}
Lei Li, Mengchao Wang, and Yuliang Wang.
\newblock Estimates of the numerical density for stochastic differential
  equations with multiplicative noise.
\newblock {\em Science China Mathematics}, 2025.

\bibitem{li2022sharp}
Lei Li and Yuliang Wang.
\newblock A sharp uniform-in-time error estimate for stochastic gradient
  {L}angevin dynamics.
\newblock {\em CSIAM Transactions on Applied Mathematics}.

\bibitem{li2020random}
Lei Li, Zhenli Xu, and Yue Zhao.
\newblock A random-batch {M}onte {C}arlo method for many-body systems with
  singular kernels.
\newblock {\em SIAM Journal on Scientific Computing}, 42(3):A1486--A1509, 2020.

\bibitem{liang2022random}
Jiuyang Liang, Pan Tan, Liang Hong, Shi Jin, Zhenli Xu, and Lei Li.
\newblock A random batch {E}wald method for charged particles in the
  isothermal--isobaric ensemble.
\newblock {\em The Journal of Chemical Physics}, 157(14), 2022.

\bibitem{lipmanflow}
Yaron Lipman, Ricky~TQ Chen, Heli Ben-Hamu, Maximilian Nickel, and Matthew Le.
\newblock Flow matching for generative modeling.
\newblock In {\em The Eleventh International Conference on Learning
  Representations}, 2023.

\bibitem{metropolis1953equation}
Nicholas Metropolis, Arianna~W Rosenbluth, Marshall~N Rosenbluth, Augusta~H
  Teller, and Edward Teller.
\newblock Equation of state calculations by fast computing machines.
\newblock {\em The journal of chemical physics}, 21(6):1087--1092, 1953.

\bibitem{milstein2004stochastic}
Grigori~N Milstein and Michael~V Tretyakov.
\newblock {\em Stochastic numerics for mathematical physics}, volume~39.
\newblock Springer, 2004.

\bibitem{mou2022improved}
Wenlong Mou, Nicolas Flammarion, Martin~J Wainwright, and Peter~L Bartlett.
\newblock Improved bounds for discretization of {L}angevin diffusions:
  Near-optimal rates without convexity.
\newblock {\em Bernoulli}, 28(3):1577--1601, 2022.

\bibitem{neal2011mcmc}
Radford~M Neal et~al.
\newblock {MCMC} using {H}amiltonian dynamics.
\newblock {\em Handbook of Markov chain Monte Carlo}, 2(11):2, 2011.

\bibitem{ninomiya2009new}
Mariko Ninomiya and Syoiti Ninomiya.
\newblock A new higher-order weak approximation scheme for stochastic
  differential equations and the {R}unge--{K}utta method.
\newblock {\em Finance and Stochastics}, 13(3):415--443, 2009.

\bibitem{ninomiya2008weak}
Syoiti Ninomiya and Nicolas Victoir.
\newblock Weak approximation of stochastic differential equations and
  application to derivative pricing.
\newblock {\em Applied Mathematical Finance}, 15(2):107--121, 2008.

\bibitem{nishimura2020discontinuous}
Akihiko Nishimura, David~B Dunson, and Jianfeng Lu.
\newblock Discontinuous {H}amiltonian {M}onte {C}arlo for discrete parameters
  and discontinuous likelihoods.
\newblock {\em Biometrika}, 107(2):365--380, 2020.

\bibitem{parulekar2025efficient}
Advait Parulekar, Litu Rout, Karthikeyan Shanmugam, and Sanjay Shakkottai.
\newblock Efficient approximate posterior sampling with annealed {L}angevin
  {M}onte {C}arlo.
\newblock {\em arXiv preprint arXiv:2508.07631}, 2025.

\bibitem{raginsky2017non}
Maxim Raginsky, Alexander Rakhlin, and Matus Telgarsky.
\newblock Non-convex learning via stochastic gradient {L}angevin dynamics: a
  nonasymptotic analysis.
\newblock In {\em Conference on Learning Theory}, pages 1674--1703. PMLR, 2017.

\bibitem{rezende2015variational}
Danilo Rezende and Shakir Mohamed.
\newblock Variational inference with normalizing flows.
\newblock In {\em International Conference on Machine Learning}, pages
  1530--1538. PMLR, 2015.

\bibitem{robbins1951stochastic}
Herbert Robbins and Sutton Monro.
\newblock A stochastic approximation method.
\newblock {\em The Annals of Mathematical Statistics}, 22(3):400--407, 1951.

\bibitem{robert2004monte}
Christian~P Robert and George Casella.
\newblock {\em Monte {C}arlo statistical methods}.
\newblock Springer, 2nd edition, 1999.

\bibitem{roberts1998optimal}
Gareth~O Roberts and Jeffrey~S Rosenthal.
\newblock Optimal scaling of discrete approximations to {L}angevin diffusions.
\newblock {\em Journal of the Royal Statistical Society: Series B (Statistical
  Methodology)}, 60(1):255--268, 1998.

\bibitem{roberts1996exponential}
Gareth~O Roberts and Richard~L Tweedie.
\newblock Exponential convergence of {L}angevin distributions and their
  discrete approximations.
\newblock {\em Bernoulli}, 2(4):341--363, 1996.

\bibitem{rossky1978brownian}
Peter~J Rossky, Jimmie~D Doll, and Harold~L Friedman.
\newblock Brownian dynamics as smart {M}onte {C}arlo simulation.
\newblock {\em The Journal of Chemical Physics}, 69(10):4628--4633, 1978.

\bibitem{songscore}
Yang Song, Jascha Sohl-Dickstein, Diederik~P Kingma, Abhishek Kumar, Stefano
  Ermon, and Ben Poole.
\newblock Score-based generative modeling through {S}tochastic differential
  equations.
\newblock In {\em International Conference on Learning Representations}, 2021.

\bibitem{talay1990expansion}
Denis Talay and Luciano Tubaro.
\newblock Expansion of the global error for numerical schemes solving
  stochastic differential equations.
\newblock {\em Stochastic Analysis and Applications}, 8(4):483--509, 1990.

\bibitem{teh2016consistency}
Yee~Whye Teh, Alexandre~H Thiery, and Sebastian~J Vollmer.
\newblock Consistency and fluctuations for stochastic gradient {L}angevin
  dynamics.
\newblock {\em The Journal of Machine Learning Research}, 17(1):193--225, 2016.

\bibitem{vempala2019rapid}
Santosh Vempala and Andre Wibisono.
\newblock Rapid convergence of the unadjusted langevin algorithm: Isoperimetry
  suffices.
\newblock In {\em Advances in Neural Information Processing Systems},
  volume~32, pages 8092--8104, 2019.

\bibitem{villani2008optimal}
C{\'e}dric Villani et~al.
\newblock {\em Optimal transport: old and new}, volume 338.
\newblock Springer, 2008.

\bibitem{wang2024weak}
Xiaojie Wang, Yuying Zhao, and Zhongqiang Zhang.
\newblock Weak error analysis for strong approximation schemes of {SDE}s with
  super-linear coefficients.
\newblock {\em IMA Journal of Numerical Analysis}, 44(5):3153--3185, 2024.

\bibitem{welling2011bayesian}
Max Welling and Yee~W Teh.
\newblock Bayesian learning via stochastic gradient {L}angevin dynamics.
\newblock In {\em Proceedings of the 28th International Conference on Machine
  Learning (ICML-11)}, pages 681--688, 2011.

\bibitem{wu2022minimax}
Keru Wu, Scott Schmidler, and Yuansi Chen.
\newblock Minimax mixing time of the {M}etropolis-adjusted {L}angevin algorithm
  for log-concave sampling.
\newblock {\em Journal of Machine Learning Research}, 23(270):1--63, 2022.

\bibitem{xifara2014langevin}
Tatiana Xifara, Chris Sherlock, Samuel Livingstone, Simon Byrne, and Mark
  Girolami.
\newblock Langevin diffusions and the {M}etropolis-adjusted {L}angevin
  algorithm.
\newblock {\em Statistics \& Probability Letters}, 91:14--19, 2014.

\bibitem{zhang2025analysis}
Matthew~S Zhang.
\newblock Analysis of {L}angevin midpoint methods using an anticipative
  {G}irsanov theorem.
\newblock {\em arXiv preprint arXiv:2507.12791}, 2025.

\end{thebibliography}

\end{document}